\documentclass[11pt,a4paper]{article}
\usepackage{fancyhdr}
\usepackage{amsmath}
\usepackage{amsthm}
\usepackage{amssymb}
\usepackage{wasysym}
\usepackage{paralist}
\usepackage{makeidx}
\usepackage[all]{xy}
\usepackage{mathdots}
\usepackage{yhmath}
\usepackage{mathtools}

\usepackage{young}
\usepackage[vcentermath]{youngtab}

\usepackage{graphicx}
\usepackage{epstopdf}

\input xy
\newcommand{\nc}{\newcommand}
 \nc{\cl}{\centerline}

 \nc{\SL}{{\rm SL}}
 \nc{\hatQ}{{\hat Q}}
 \nc{\sgn}{{\rm sgn}}
 \nc{\ad}{{\rm ad}}
 
  \nc{\idx}{{\rm index}}
 
 \nc{\Mat}{{\rm Mat}}
 
  \nc{\Ann}{{\rm Ann}}
 
 \nc{\Loewy}{{\rm Loewy}}
 
 \nc{\efp}{{\Bbb F}_p}
 
 \nc{\tildeN}{{\tilde N}}
 
 \nc{\tildeC}{{\tilde C}}
 
 \nc{\tildeA}{{\tilde A}}
 
 \nc{\tildeW}{{\tilde W}}
 
 \nc{\tildeS}{{\tilde S}}

 \newcommand{\Par}{{\rm Par}}
 
\nc{\baru}{{\overline u}}

\nc{\barA}{{\bar A}}

\nc{\barq}{{\overline q}}

\nc{\barR}{{\bar R}}

 \nc{\Orb}{{\rm Orb}}

 \nc{\hatsigma}{{\hat \sigma}}

 \nc{\hatpi}{{\hat \pi}}
 
  \nc{\hatzeta}{{\hat \zeta}}

 \nc{\Ocal}{{\mathcal O}}

 \nc{\M}{\mathfrak{m}}
 
 \nc{\seee}{\mathbb C}
 
 \nc{\varleq}{\preccurlyeq}

  \newcommand{\DEnd}{{\rm DEnd}}
 
 \nc{\hatlambda}{{\hat\lambda}}
 
 \nc{\hatphi}{{\hat \phi}}
 
 \nc{\daggerlambda}{{\lambda^\dagger}}

    \nc{\barr}{{\bar r}}
  \nc{\bart}{{\bar t}}
  
    \nc{\barsigma}{{\bar \sigma}}

  \newcommand{\resp}{{resp.\,}}

  \newcommand{\Hec}{{\rm Hec}}

   \newcommand{\op}{{\rm op}}

\nc\diag{{\rm diag}}
\renewcommand{\vert}{{\,|\,}}
\nc{\hatL}{{\hat L}}

\nc{\barE}{{\bar   E}}
\nc{\D}{{\mathcal D}}
\nc{\E}{{\mathcal E}}
\nc{\F}{{\mathcal F}}
\nc{\FF}{{\mathcal F}}
\nc{\I}{{\mathcal I}}
\nc{\even}{{\rm e}}
\nc{\ep}{\epsilon}
\nc{\odd}{{\rm o}}
\nc{\Coker}{{\rm Coker}}
\nc{\olE}{{\overline E}}
\nc{\indBG}{{\rm ind}_B^G\,}
\nc{\indHG}{{\rm ind}_H^G\,}
\renewcommand{\O}{{\mathcal O}}

\nc{\que}{{\mathbb Q}}
\nc{\barlambda}{{\bar\lambda}}
\nc{\barmu}{{\bar\mu}}
\nc{\barnu}{{\bar\nu}}
\nc{\bartau}{{\bar\tau}}
\nc{\barm}{{\overline  m}}
\nc{\divind}{{\rm div.ind}}
\nc{\tl}{{\tilde{\lambda}}}
\nc{\dar}{\downarrow}

\nc{\eno}{{\mathbb N}_0}

\nc{\Sym}{{\rm Sym}}
\nc{\Symm}{{\rm Sym}}

\newcommand{\q}{\quad}
\newcommand{\de}{\delta}
\newcommand{\real}{{\rm real}}

\renewcommand{\mod}{{\rm mod}}

\newcommand{\Sp}{{\rm Sp}}

\newcommand{\bs}{\bigskip}

\renewcommand{\vert}{\,|\,}

\renewcommand{\sgn}{{\rm sgn}}
\newcommand{\e}{{\rm e}}
\newcommand{\reg}{{\rm reg}}

\xyoption{all}

\newcommand{\ind}{{\rm ind}}
\renewcommand{\vert}{\,|\,}
 \newcommand{\tbw}{\textstyle\bigwedge}
\newcommand{\zed}{{\mathbb Z}}

\newcommand{\End}{{\rm End}}
\newcommand{\Hom}{{\rm Hom}}

\renewcommand{\mod}{{\rm mod}}

\nc{\barB}{{\overline B}}
\nc{\barb}{{\overline b}}

\renewcommand{\mod}{{\rm{mod}}}

\nc{\geom}{{\rm geom}}
\nc{\rep}{{\rm rep}}

\newtheorem{definition}{Definition}[section]
\newtheorem{proposition}[definition]{Proposition}

\newtheorem{lemma}[definition]{Lemma}
\newtheorem{corollary}[definition]{Corollary}
\newtheorem{example}[definition]{Example}

\newtheorem{remark}[definition]{Remark}

\begin{document}

%endotextone.tex

\centerline{\bf  Double centralisers and annihilator ideals of}

\centerline{ \bf  Young permutation modules}

\bigskip

\centerline{Stephen Donkin}

\bigskip

{\it Department of Mathematics, University of York, York YO10 5DD}

\medskip

\centerline{\tt stephen.donkin@york.ac.uk}

\bs

\centerline{24  July   2020}

\bs

\section*{Abstract}  We study  the double centraliser  algebras  and annihilator ideals of certain  permutation modules for symmetric groups and their quantum analogues.  Of particular interest is the behaviour of these algebras and ideals under base change.  By exploiting connections with  polynomial representation theory over a field  we establish the double centraliser property and dimension invariance of annihilator ideals  over fields.   By elementary localisation techniques we  obtain the corresponding results over arbitrary rings.   The modules considered include the permutation modules relevant to partition algebras, as  studied in the  recent paper by Bowman, Doty and Martin, \cite{BDM},    and we obtain in particular  their main results on the double centraliser property  in this context.  Much  of our  the development is equally  valid for the corresponding modules over Hecke algebras of type $A$ and, for the most part,  we give it at this level of generality.

%\bs\bs\bs

\section{Introduction}    
Let $R$ be a ring and let $P$ be a left $R$-module. We have the annihilator ideal 
$$\Ann_R(P)=\{ r \in R  \vert r p=0 \hbox{ for all } p\in P\}$$  
 and the endomorphism (or centraliser)  ring 
$$\End_R(P)=\{\alpha: P\to  P \vert \alpha(r p)=r \alpha(p) \hbox{ for all }r\in R, p\in P\}$$
and thus, regarding $P$ as an $\End_R(P)$-module, the double endomorphism (or double centraliser) ring  
$$\DEnd_R(P)=\{ \sigma : P\to P \vert  \sigma\circ \alpha=\alpha\circ \sigma \hbox{ for all } \alpha \in \End_R(P)\}.$$

\q We have the natural map $\nu:R\to \DEnd_R(P)$ given by the action, i.e. $\nu(r)(p)=rp$, for $r\in R$, $p\in P$, and the complex  
$$0\to \Ann_R(P)\to R\to \DEnd_R(P)\to 0$$
which is exact except possibly at $\DEnd_R(P)$. We are interested in cases in which this  is exact, i.e. in which $\nu$ is surjective. When this happens we say that $P$ is a double centraliser module, or that $P$ has the double centraliser property.

\q  We are particularly interested in base change properties, and establishing the double centraliser property via this route.  In general the ring $R$ may be an algebra over a commutative ring $A$. Given an $A$-algebra $B$ we can then form the $A$-algebra $R_B=B\otimes_A R$ and the $R_B$-module $P_B=B\otimes_A P$.  We have the natural map on annihilators 
$$\Phi_B: B\otimes_A \Ann_R(P)\to \Ann_{R_B}(P_B).$$

\q For an $A$-algebra $B$ we write $B^\op$ for the opposite algebra. We also consider base change on endomorphism algebras. We have the natural map of  $A$-algebras
$$\Psi_B:B^\op \otimes_A \End_R(P)\to \End_{R_B}(P_B)$$
satisfying $\Psi_B(b\otimes\alpha)(b'\otimes p)=b'b\otimes \alpha(p)$, for $b,b'\in B$, $\alpha\in \End_R(P)$, $p\in P$.

\q We are interested in cases in which $\Phi_B$ and $\Psi_B$ are isomorphisms.

\q Our main focus is the case in which $R$ is a group algebra of a finite group and $P$ is a permutation module  (and the Hecke algebra analogue of this situation).  We take $A=\zed$, $G$ a finite group, $X$ a finite $G$-set, $R=\zed G$ and $P=\zed X$.  In this case the endomorphism algebra $\Lambda=\End_{\zed G}(\zed X)$ has a basis given by orbits of $G$ on $X\times X$ and  this gives rise to a basis over any coefficient ring. We may then study  base change for $P$ regarded as a module over $\Lambda$, and hence study the double centraliser of the permutation module over a group algebra.

\q Our primary concern is to show that certain permutation modules for  group algebras of symmetric groups   (and their quantised analogues over Hecke algebras)  have the double centraliser  property.  There are two aspects to this. We prove this over fields by using connections with the polynomial representation theory of general linear groups (and their  quantised analogues).  The result is  deduced over the integers and then over arbitrary rings by  elementary general considerations.   

\q We deal with these general elementary considerations in the first part of the paper.   They are closely connected with question of whether the annihilator of a module $P$ as above is stable under base change,  and the arguments used in both cases are extremely similar.  We assume, where convenient that $A$ is a Dedekind domain. We convert the problems into matrix problems for $A$-modules and show that these are compatible with localisation. In this way we are reduced to the case in which $A$ is a discrete valuation ring. The main result in the first  part (see Corollary 2.13 and Proposition 3.3)  is that,  under  certain assumptions, if base change holds  from the regular module $A$  to the module $A/\M$ and the double centraliser property holds for the $R/\M R$-module $P/\M P$,  for each maximal ideal  $\M$, then $P$ has the double centraliser property and  the base change property on ideals and endomorphism algebras.  In particular the double centraliser property holds over all $A$-algebras obtained by base change from $R$ (so over all coefficient rings if  $A=\zed$) if it holds over the residue fields.

\q The second part of the paper is more technical.  We establish the double centraliser property for many tilting modules for the quantised general linear groups (equivalently for the  $q$ Schur algebras). This  information is then transferred to the setting of representations of type $A$ Hecke algebras via the Schur functor and we deduce the double centraliser property and base change on double centralisers and annihilators  for many Young permutation modules, including those relevant to the partition algebras.  In the penultimate section we give an example of a  Young permutation module that  does not   have the double centraliser property or base change on annihilator ideals. In the final section we make some remarks on cell ideals and annihilators in the spirit of  \cite{Haerterich}, \cite{DotNym}, \cite{XX} and show that for many Young permutation modules the annihilator ideal is  a cell ideal.

\bs\bs\bs\bs

%\newpage

\centerline{\bf \LARGE  Part I: Base Change}

\bs\bs

\section{Localisation and base change for  annihilator ideals and centraliser  algebras.}

\q All modules will be left unless otherwise specified.  Let  $A$ be  a Dedekind domain or a field,   let $\Lambda$ be an $A$-algebra and $P$ a $\Lambda$-module.  We assume that  $\Lambda$ and $P$ are finitely generated and free over $A$.  We fix an $A$-basis $\lambda_1,\ldots,\lambda_n$ of $\Lambda$ and an $A$-basis $p_1,\ldots,p_d$ of $P$.  (In the applications, $P$ will be a module over an $A$-algebra $R$ and $P$ will be regarded as a module for $\Lambda=\End_R(P)$.) 

\q We have the structure constants $c^t_{is}\in A$ defined by the equations
$$\lambda_i p_s=\sum_{t=1}^d c^t_{is} p_t$$
for $1\leq i\leq n$, $1\leq s\leq d$. 

\q Let $B$ be an $A$-algebra.   An element $q$  of $\Lambda_B=B\otimes_A \Lambda$ may be uniquely written in the form $q=\sum_{i=1}^n b_i\otimes\lambda_i$.  The condition for $q$ to belong to the annihilator  $\Ann_{\Lambda_B}(P_B)$ of $P_B=B\otimes_A P$  is that $q(b\otimes p_s)=0$, for all $b\in B$, $1\leq s\leq d$, i.e.   $\sum_{t=1}^d \sum_{i=1}^n  b_ib\otimes c^t_{is}p_t=0$, i.e.
$$\sum_{i=1}^n c^t_{is}b_i=0\eqno{(1)}$$
 for all $1 \leq s,t\leq d$.
 
 \q It is convenient to generalise this condition to an arbitrary $A$-module.  For a module $M$ and $r\geq 0$  we write $M^r$ or $M^{(r)}$ for $M\oplus \cdots\oplus M$, the direct sum of $r$ copies of $M$.

 \begin{definition} Let $M$ be an $A$-module. We shall say that $(m_1,\ldots,m_n)\in M^n$ is an admissible vector  if 
 $$\sum_{i=1}^n c^t_{is}m_i=0  \eqno{(2)}$$ 
 for all $1\leq s,t\leq d$.  
 \end{definition}
 
\q  We write $M^n_\ad$ for the $A$-submodule of $M^n$ consisting of all admissible vectors.

\q We have the natural $A$-module isomorphism $A^n\otimes_A M\to M^n$, which restricts to an $A$-module homomorphism $\phi_M:A^n_\ad\otimes_A  M\to M^n_\ad$. Moreover, $A^n_\ad$ is a pure submodule of $A^n$, i.e.  $A^n/A^n_\ad$ is torsion free, and hence flat, so that $\phi_M$ is injective. We shall call an element of $M^n$ a realisable vector (or vector-realisable)  if it belongs to the image of $\phi_M$, i.e.  $y\in M^n$ is realisable if for some $r\geq 1$ there exist elements $(a_{i1},\ldots,a_{in})\in A^n_\ad$ , for $1\leq i\leq r$, and $m_1,\ldots,m_r\in M$ such that $y=\sum_{i=1}^r (a_{i1}m_i,\ldots,a_{in}m_i)$.
We write $M^n_\real$ for the submodule of $M^n$ consisting of realisable vectors and  say that $M$ is vector-realisable if $M^n_\ad=M^n_\real$ (i.e. if $\phi_M$ is surjective).

\begin{remark}  One may give a basis free interpretation of $M^n_\ad$ as follows.    We have the representation $\rho: \Lambda\to \End_A(P)$ and hence the induced map \\
$\theta:\Lambda\otimes_A M  \to  \End_A(P)\otimes_A M$. The $A$-module $\End_A(P)$ has basis $e_{st}$, where $e_{st}(p_u)$ is $p_s$ if $u=t$ and $0$ otherwise. Then  $\rho(\lambda_i)=\sum_{s,t} c^t_{is} e_{ts}$ and hence, for $h=\sum_i \lambda_i\otimes m_i  \in \Lambda\otimes_A M$, we have $\theta(h)=\sum_{i,s,t} c^t_{is} e_{ts}\otimes m_i$ and hence $\theta(h)=0$ if and only if $\sum_{i=1}^n c^t_{is} m_i=0$, for all $s,t$. In this way $M^n_\ad$ is identified with the kernel of the natural map $\theta: \Lambda\otimes_A M \to  \End_A(P)\otimes_A M$.
\end{remark}

\q We study the  the situation for endomorphism algebras in a similar fashion.   An element $\sigma\in \End_{\Lambda_B}(P_B)$ may be represented by a $d\times d$  - matrix  $(\sigma_{st})$, with entries in $B$,  defined by the equations
$$\sigma(1\otimes p_s)=\sum_{t=1}^d \sigma_{ts} \otimes p_t$$
for $1\leq s\leq d$.  Given such a matrix $(\sigma_{st})$ we may define a $B$-endomorphism as above.  The condition for  $\sigma$  to be  a $\Lambda_B$-endomorphism is that it commutes with the action of $\Lambda$, i.e. that $\sigma((1\otimes \lambda_i)(1\otimes p_s))=(1\otimes \lambda_i)\sigma(1\otimes p_s)$ i.e.
$\sum_{t,u} c^t_{is} \sigma_{ut}\otimes p_u=\sum_{t, u} c^u_{it} \sigma_{ts}\otimes p_u$, i.e. 
$$\sum_t c^t_{is} \sigma_{ut}=\sum_t  c^u_{it} \sigma_{ts}\eqno{(3)}$$
for all $i,s,u$.

\q It will be convenient to generalise this condition to $A$-modules. For an $A$-module $M$ we write $\Mat_d(M)$ for the additive group of $d\times d$ matrices with entries in $M$. We regard $\Mat_d(M)$ as an $A$-module with entrywise action, i.e. 
$a(m_{st})=(am_{st})$. Note that we have the  $A$-module isomorphism $\Mat_d(A)\otimes_A M\to \Mat_d(M)$, taking $(a_{st})\otimes m$ to $(a_{st}m)$, for $(a_{st})\in \Mat_d(A)$, $m\in M$.

\begin{definition} Let $M$ be an $A$-module.  We shall say that $(m_{st})\in \Mat_d(M)$ is an admissible matrix  if 
$$\sum_t  c^t_{is} m_{ut}=\sum_t  c^u_{it} m_{ts}\eqno{(4)}$$
for all $i,s,u$.
\end{definition}

\q We write $\Mat_d(M)_\ad$ for the set of admissible matrices over $M$. Note that $\Mat_d(M)_\ad$ is an $A$-submodule of $\Mat_d(M)$.   Note also  that the isomorphism $\Mat_d(A)\otimes_A M\to \Mat_d(M)$ induces a map \\
$\psi_M: \Mat_d(A)_\ad\otimes  M\to \Mat_d(M)_\ad$.  Moreover, $\Mat_d(A)_\ad$ is a pure submodule of $\Mat_d(A)$, i.e. $\Mat_d(A)/\Mat_d(A)_\ad$ is torsion free, and hence flat, so that $\psi_M$ is injective. We write $\Mat_d(M)_\real$ for the image of $\psi_M$.   We say that an element $(m_{st})$ of $\Mat_d(M)$ is a realisable matrix (or matrix-realisable)   if it belongs to the image of $\psi_M$,  i.e, if for some $r\geq 1$ there exist elements $(a^i_{st})\in \Mat_d(A)_\ad$, and $m_i\in M$, for $1\leq i\leq r$, such that  $m_{st}=\sum_{i-=1}^r a^i_{st} m_i$, for all $1\leq s,t\leq d$. We shall say that $M$ is matrix-realisable if $\Mat_d(M)_\ad=\Mat(M)_\real$ (i.e. if $\psi_M$ is surjective). 

\begin{remark}  One may give a basis free interpretation of $\Mat_d(M)_\ad$ as follows.   The endomorphism algebra $\End_A(P)$ is naturally a $\Lambda$-bimodule.  Hence $\End_A(P)\otimes_A M$ is naturally a $\Lambda$-bimodule with action $\lambda( \phi\otimes m)\mu= \lambda\phi\mu\otimes m$, for $\lambda,\mu\in \Lambda$, $m\in M$, $\phi\in \End_A(P)$.   The $A$-module $\End_A(P)$ has basis $e_{st}$, where $e_{st}(p_u)$ is $p_s$ if $u=t$ and $0$ otherwise.  For  $1\leq i\leq n$, $1\leq s,t\leq d$,  we have $\lambda_ie_{st}=\sum_u  c_{is}^ue_{ut}$ and $e_{st}\lambda_i=\sum_v c^t_{iv}e_{sv}$ and it follows that an element $\sum_{s,t}  e_{st}\otimes m_{st}$ of  $ \End_A(P)\otimes_A M$ belongs to the zeroth Hochschild cohomology group $H^0(\Lambda,\End_A(P)\otimes_A M)$ if and only if 
$$\sum_t  c^t_{is} m_{ut}=\sum_t  c^u_{it} m_{ts}$$
for all $i,s,u$. In this way we may identify $\Mat_d(M)_\ad$ and \\
$H^0(\Lambda,  \End_A(P)\otimes_A M)$.

\q We have  $\End_\Lambda(P)=H^0(\Lambda,\End_A(P))$, hence the induced map  \\
$\End_\Lambda(P)\otimes_A M \to H^0(\Lambda, \End_A(P)\otimes_A M)$,  and the question is whether this is an isomorphism.
\end{remark}

\begin{remark}   We have the natural isomorphisms   $A^n_\ad\otimes_A  A\to A^n_\ad$   and 
$\Mat_d(A)_\ad\otimes_A  A\to \Mat_d(A)_\ad$  so that the left regular module $A$ is vector-realisable and matrix-realisable. Moreover, a direct sum $M=\bigoplus_{i\in I}  M_i$ is vector-realisable (\resp matrix-realisable)  if and only if each $M_i$ is vector-realisable (\resp matrix-realisable).  In particular a free $A$-module is vector-realisable and matrix-realisable  and if $A$ is a field then every $A$-module is vector-realisable and matrix-realisable.
\end{remark}

\begin{lemma} Let $0\to L\to M\to N\to 0$ be a short exact sequence of $A$-modules. If $L$ and $N$ are vector-realisable (\resp matrix-realisable)  then so is $M$.
\end{lemma}

\begin{proof} We give the matrix version and leave the vector version to the reader. We may assume that $L$ is a submodule of $M$ and $N=M/L$.  Let $(m_{st})\in \Mat_d(M)_\ad$.  We set ${\overline  m}_{st}=m_{st}+L\in M/L$, $1\leq s,t\leq d$.  Then $(\barm_{st})\in \Mat_d(N)_\ad$.  Hence there exist a positive integer $u$, $(a^i_{st})\in \Mat_d(A)_\ad$ and $n_i\in N$, for $1\leq i\leq u$, such that $\barm_{st}=\sum_{i=1}^u a^i_{st}n_i$, for all $s,t$.  We write $n_i=m_i+L$, for some $m_i\in M$, $1\leq i\leq u$.  Then we have $m_{st}\equiv \sum_{i=1}^u a^i_{st}m_i$ (mod $L$). Putting $l_{st}=m_{st} -  \sum_{i=1}^u a^i_{st}m_i$, we have that $(l_{st})\in \Mat_d(N)_\ad$ and hence there exists a positive integer $v$, $(b^j_{st})\in \Mat_d(A)_\ad$ and $l_j\in L$, for $1\leq j\leq v$, such that $l_{st}=\sum_{j=1}^v b^j_{st}l_j$, for all $s,t$.  Then we have $m_{st}=  \sum_{i=1}^u a^i_{st}m_i+\sum_{j=1}^v b^j_{st}l_j$, which shows that $(m_{st})$ is realisable.
\end{proof}

\q The constructions  introduced above behave well with respect to localisation.  Let $S$ be a multiplicatively closed subset of $A$ (containing $1$). Then we have the $S^{-1}A$-algebra  $S^{-1}\Lambda$, with   
$S^{-1}A$ basis $\frac{\lambda_i}{1}$, 
$1\leq i\leq n$,  and the $S^{-1}P$-module with $S^{-1}A$-basis   $\frac{p_s}{1}$,  $1\leq s\leq d$.  We consider admissible vectors and matrices for  an $S^{-1}A$-module with respect to these bases and the   structure constants $\frac{c^t_{is}}{1}$.

\begin{lemma} Let $S$ be a multiplicatively closed subset of $A$  and let $M$ be an $A$-module.

(i) The  inclusions of  $M^n_\ad$ and   $M^n_\real$ into $M^n$  induce isomorphisms\\
 $$S^{-1} M^n_\ad\to (S^{-1}M)^n_\ad    \   \hbox{ and }$$
 $$   S^{-1}M^n_\real \to (S^{-1}M)^n_\real. $$

(ii) The  inclusions of  $\Mat_d(M)_\ad$ and   $\Mat_d(M)_\real$ into  $\Mat_d(M)$  induce isomorphisms\\
 $$S^{-1}\Mat_d(M)_\ad\to \Mat_d(S^{-1}M)_\ad    \   \hbox{ and }$$
 $$   S^{-1}\Mat_d(M)_\real \to \Mat_d(S^{-1}M)_\real. $$

(iii) If $M$ is a vector-realisable (\resp matrix-realisable)    $A$-module  then $S^{-1}M$ is a vector-realisable (\resp matrix-realisable)    $S^{-1}A$-module.
\end{lemma}

\q For an $A$-module $M$ and a maximal ideal $\M$ of $A$  we write $M_\M$ for the localisation $S^{-1}M$, where $S=A\backslash \M$.

\begin{proposition} Let $M$ be an $A$-module.    Then   $M$ is a vector-realisable (\resp matrix-realisable)  $A$-module  if and only if $M_\M$ is a vector-realisable (\resp matrix-realisable)  $A_\M$-module for every maximal ideal $\M$ of $A$.
\end{proposition}

\begin{proof}   We give the matrix version.   

\q  If $M$ is matrix-realisable $A$-module and $\M$ is a maximal ideal of $A$  then $M_\M$ is a matrix-realisable $A_\M$-module by Lemma  2.7(iii).

\q Now suppose that for every maximal ideal $\M$ of $A$  the $A_\M$-module $M_\M$ is matrix-realisable.  Let $N=\Mat_d(M)_\ad/\Mat_d(M)_\real$. The short exact sequence 
$$0\to \Mat_d(M)_\real \to \Mat_d(M)_\ad \to N\to 0$$
 gives rise to  an exact sequence 
 $$0\to  (\Mat_d(M)_\real)_\M  \to (\Mat_d(M)_\ad)_\M\to N_\M\to 0$$
  and hence, by Lemma 2.7(ii),  an  exact sequence 
$$0\to \Mat_d(M_\M)_\real\to \Mat_d(M_\M)_\ad \to N_\M\to 0.$$
 Thus we have $N_\M=0$ for every maximal ideal $\M$ of $A$ and hence, by \cite[Proposition 3.8]{AM},  $N=0$, i.e. $\Mat_d(M)_\ad=\Mat_d(M)_\real$.
\end{proof}

\q We obtain the following criterion for all $A$-modules to be  realisable in terms of residue fields.

\begin{proposition}  If   $A/\M$ is a  vector-realisable (\resp matrix-realisable)  $A$-module, for every maximal ideal $\M$ of $A$,  then every $A$-module is vector-realisable (\resp  matrix-realisable).
\end{proposition}

\begin{proof}    As usual we give the matrix version.  By Proposition 2.8 and Remark 2.5,  we may assume that $A$ is local, i.e. $A$ is a discrete valuation ring, with maximal ideal $\M$, say.   Let $M$ be an $A$-module and let 
$(m_{st})\in \Mat_d(M)_\ad$. We let $M_0$ be the $A$-submodule generated by the elements $m_{st}$, $1\leq s,t\leq d$. If $M_0$ is  matrix-realisable then the matrix  $(m_{st})$ is realisable. Thus we may assume that $M$ is finitely generated.

\q Let $T(M)$ be the torsion submodule of $M$. Then $M/T(M)$ is a free $A$-module and so is  matrix-realisable by Remark 2.5.  Hence, by Lemma 2.6, it suffices to prove that $T(M)$ is matrix-realisable, i.e. we may assume that $M$ is a finitely generated  torsion module.

\q We choose a positive integer  $k$ such that $\M^k M=0$.  If $\M^{i-1}M /\M^i M$ is  matrix-realisable, for $1\leq i\leq k$, then so is $M$, by Lemma 2.6.  So we may assume $\M M=0$, and so $M$ is an $A/\M$ vector space.  Hence $M$ is matrix-realisable as an $F=A/\M$-module.  

\q Consider $(m_{st})\in \Mat_d(M)_\ad$.   Then, for some $u$ there exists admissible matrices  $(f^i_{st})$, with entries in $F$,  and $m_i\in M$, for $1\leq i\leq u$, such that   $m_{st}=\sum_{i=1}^u f^i_{st} m_i$, for all $1\leq s,t\leq d$. By hypothesis $F$ is a matrix-realisable  $A$-module, so that, for $1\leq i\leq r$, there exists an  integer $v_i$ and admissible matrices  $(a^{ij}_{st})$ (with entries in $A$)  and $g_{ij}\in F$, for $1\leq j\leq v_i$, such that $f^i_{st}=\sum_{j=1}^{v_i}  a^{ij}_{st} g_{ij}$,  for $1\leq i\leq r$. Hence, for all $1\leq s,t\leq d$,  we have
  $$m_{st}=\sum_{i=1}^u f^i_{st} m_i=\sum_{i,j} a^{ij}_{st}g_{ij}m_i=\sum_{i,j} a^{ij}_{st} x_{ij}$$
  where $x_{ij}=g_{ij} m_i$, which shows that $(m_{st})$ is realisable, and completes the proof.
 \end{proof}

 \begin{remark}  Suppose that $\phi:A\to F$ is a homomorphism from $A$ to a field $F$  and $K$ is a field extension.  We note that $F$ is vector-realisable (\resp 
 matrix-realisable) if and only if $K$ is.  We shall deal with the matrix version.  We choose a basis $b_y$, $y\in Y$, of $K$ as an $F$-space, such that $b_{y_0}=1$ for some $y_0\in Y$.  
 
 \q Suppose that $F$ is matrix realisable and let $(k_{st})\in \Mat_d(K)_\ad$.  Then  for some finite subset $Y_1$ of $Y$ we have, for all $1\leq s,t\leq d$,  equations
 $$k_{st}=\sum_{y\in Y_1} f_{st}^y b_y$$
 for some $f_{st}^y\in F$. Moreover, for each $y\in Y_1$, the matrix $(f^y_{st})$ belongs to $\Mat_d(F)_\ad$.  Hence, for each $y\in Y_1$, there exists a positive integer $h_y$, elements $\alpha^{yi}_{st}\in F$, such that, for fixed $y$ and $i$ the matrix $(\alpha^{yi}_{st})$ belongs to $\Mat_d(F)_\ad$,  and elements $x_{yi} \in K$, for $1\leq i\leq h_y$, such that 
 $$f^y_{st}=\sum_{i=1}^{h_y} \alpha^{yi}_{st} x_{yi}.$$
 Hence we have 
 $$k_{st}=\sum_{y\in Y_1, 1\leq i\leq h_y}\alpha^{yi}_{st} x_{yi} b_y=\sum_{y\in Y_1, 1\leq i\leq h_y} \alpha^{yi}_{st} t_{yi}$$
 where $t_{yi}=x_{yi}b_y$,  which shows that $(k_{st})$ is realisable, and hence $K$ is a matrix-realisable  $A$-module. 
 
 \q  Conversely suppose  that $K$ is a realisable  $A$-module. We consider $(f_{st})\in \Mat_d(F)_\ad$. Then $(f_{st})\in \Mat_d(K)_\ad$.  Hence there exists a positive integer $h$, elements $(a_{st}^i)$ of $\Mat_d(A)_\ad$, and $k_1,\ldots,k_h\in K$ such that 
 $$f_{st}=\sum_{i=1}^h a^i_{st} k_i$$
 for $1\leq s,t\leq d$.  We write $k_i=\sum_{y\in Y} \beta_{iy} b_y$, as an $F$-linear combination of the basis elements $b_y$, for $1\leq i\leq h$.  Thus  we get 
  $$f_{st}=\sum_{i=1,\ldots,h,y\in Y} a^i_{st} \beta_{iy} b_y$$
  and so, equating coefficients of $y_{b_0}=1$ we get 
 $$f_{st}=\sum_{i=1}^h a^i_{st} \beta_{iy_0}$$
 for $1\leq s,t\leq d$, which shows that $(f_{st})$ is realisable and hence $F$ is a realisable $A$-module. 
 
 \q It follows that, in the context of Proposition 2.9, every $A$-module is vector-realisable (\resp matrix-realisable) provided that, for each maximal ideal $\M$ of $A$ there is a field $F$ containing $A/\M$ which is vector-realisable (\resp matrix realisable) as an $A$-module.
 \end{remark}

 \begin{remark} Suppose that $A=F[t,t^{-1}]$ is the  ring of Laurent polynomials over a field $F$. Let $K$ be a field extension of $F$. Then we obtain, by base change $A_K=K[t,t^{-1}]$, the $A_K$-algebra $\Lambda_K=K\otimes_F\Lambda$ and the $\Lambda_K$-module $P_K=K\otimes_F P$.  Then  by arguments similar to those of the above remark, we see that an $A$-module $M$ is  vector-realisable (\resp matrix realisable) if and only if   the $A_K$-module $M_K=K\otimes_F M$   is a vector-realisable (\resp matrix realisable) $A_K$-module.  We leave the details to the interested reader. 
 \end{remark}

 \begin{remark}   Let $B$ be an $A$-algebra.  The natural map 
 $$\Phi_B: B\otimes_A \Ann_\Lambda(P)\to \Ann_{\Lambda_B}(P_B)$$
 admits a factorisation of $A$-module maps 
 $$B\otimes_A \Ann_\Lambda(P)  \xrightarrow{\alpha}    A^n_\ad\otimes_A B \xrightarrow{\beta} B^n_\ad \xrightarrow{\gamma}  \Ann_{\Lambda_B}(P_B)$$
 where $\alpha:B\otimes_A \Ann_\Lambda(P)  \to    A^n_\ad\otimes_A B$ satisfies \\ 
 $\alpha(b\otimes\sum_{i=1}^n a_i\lambda_i)=(a_1,\ldots,a_n)\otimes b$, for $b\in B$, $\sum_i a_i\lambda_i \in \Ann_\Lambda(P)$, where \\
 $\beta: A^n_\ad\otimes_A B \to  B^n_\ad$ satisfies $\beta((a_1,\ldots,a_n)\otimes b)=(a_1b,\ldots,a_nb)$, for $(a_1,\ldots,a_n)\in A^n_\ad$, $b\in B$, and where 
 $\gamma:B^n_\ad  \to  \Ann_{\Lambda_B}(P_B)$ satisfies $\gamma(b_1,\ldots,b_n)=\sum_{i=1}^n b_i\otimes \lambda_i$, for $(b_1,\ldots,b_n)\in B^n_\ad$.
 
 \q Moreover, the maps $\alpha$ and $\gamma$ are isomorphisms so that $\Phi_B$ is an isomorphism if and only if  $\beta$ is an isomorphism i.e.  if and only if $B$ is a vector-realisable $A$-module.

 \q Likewise, the natural map 
 $$\Psi_B:B^\op \otimes_A \End_\Lambda(P)\to \End_{\Lambda_B}(P_B)$$
  admits a factorisation of $A$-module maps 
  $$B^\op \otimes_A \End_\Lambda(P)\xrightarrow{\delta} \Mat_d(A)_\ad\otimes_A B \xrightarrow{\ep} \Mat_d(B)_\ad \xrightarrow{\eta} \End_{\Lambda_B}(P_B)$$
 where $\de:  B^\op\otimes_A \End_A(P)\to \Mat_d(A)_\ad\otimes_A B$ satisfies $\de(b\otimes \sigma)=(\sigma_{st})\otimes b$, for $b\in B$, $\sigma=(\sigma_{st})\in  \End_A(P)$, where $\ep: 
  \Mat_d(A)_\ad\otimes_A B \to \Mat_d(B)_\ad$ satisfies $\ep((\sigma_{st})\otimes b)=(\sigma_{st}b)$,  for $(\sigma_{st})\in \Mat_d(A)_\ad$, $b\in B$, and where $\eta:
  \Mat_d(B)_\ad \to  \End_{\Lambda_B}(P_B)$ satisfies $\eta((b_{st}))(b\otimes p_u)=\sum_t bb_{tu}\otimes p_t$,  for $(b_{st})\in \Mat_d(B)_\ad$, $b\in B$, $1\leq u\leq d$.
  
   \q Moreover, the maps $\delta$ and $\eta$ are isomorphisms so that $\Psi_B$ is an isomorphism if and only if $\ep$ is, i.e.  if and only if $B$ is a matrix realisable $A$-module.
  \end{remark}
  
  \q From Proposition 2.9 and Remark 2.10 we thus get the following result.

 \begin{corollary}  Suppose that $\Lambda$ is an $A$-algebra which is free of finite rank over $A$ and  that $P$ is a $\Lambda$-module which is free of finite rank over $A$. 
 
 (i) Suppose that for every maximal idea $\M$ there exists a homomorphism from $A$ to a field  $F$ with kernel $\M$ such that  natural map 
 $F\otimes_A \Ann_\Lambda(P)   \to \Ann_{\Lambda_F}(P_F)$ is surjective. Then for any $A$-algebra $B$ the natural map
 $$\Phi_B: B\otimes_A \Ann_\Lambda(P)\to \Ann_{\Lambda_B}(P_B)$$
 is an isomorphism.

 (ii)  Suppose that for every maximal idea $\M$ there exists a homomorphism from $A$ to a field  $F$ with kernel $\M$ such that  natural map  $F\otimes_A \End_\Lambda(P)\to \End_{ \Lambda_F}(P_F)$ is surjective. Then for any $A$-algebra $B$ the natural map
 $$\Psi_B:B^\op\otimes_A \End_\Lambda(P)\to \End_{\Lambda_B}(P_B)$$
 is an isomorphism.
 \end{corollary}

%\newpage

\section{Double Centralisers}

\q  We continue to assume that $A$ is a Dedekind domain or a field,  that $R$ is an  $A$-algebra and  that $P$ an $R$ module.

\bs

\it Hypotheses:   \rm We shall require the following properties: 

(i)   the $A$-algebra $R$  and the $R$-module $P$ are free of finite rank over $A$;

(ii)  $\End_R(P)$ is  free of finite rank over $A$; and 

(iii)   $\End_R(P)$  has the base change property on endomorphism algebras, i.e. the map 
$\Psi_B: B^\op \otimes_A \End_R(P) \to \End_{R_B}(P_B)$
 is an isomorphism, for every $A$-algebra $B$.
 
 \bs
 
 \q When these properties hold we shall say that the triple $(A,R,P)$ satisfies the standard hypotheses.

\q Note that, by Corollary 2.13,  given (i) and (ii), the final property (iii)   holds provided that it holds over residue fields,  i.e.  provided that   for each residue field $F$   the natural map $F\otimes_A \End_R(P)\to \End_{R_F}(P_F)$ is surjective, and therefore an isomorphism.

\begin{example} (i)  Let $G$ be a finite group and let $X$ be a  finite $G$-set.    Let $S$ be  a commutative ring,  Then we have the  permutation module  $SX$. For an orbit $\O$ of $G$ on $X\times X$ we have the $SG$-module endomorphism $\alpha_\O$ of $SX$ satisfying
$\alpha_\O(x)=\sum_{(y,x)\in \O} y$,  and these elements form an $S$-basis of $\End_{SG}(SX)$.   In particular $\Lambda=\End_{\zed G}(\zed X)$ is free of finite rank over $\zed$.  It follows that the $\zed G$-module $\zed X$ has the base change property on endomorphism algebras and that $(\zed,\zed G,\zed X)$ satisfies the standard hypotheses.
\end{example}

\begin{example} Let $F$ be a field and $A=F[t,t^{-1}]$.  Let $n$ be a positive integer and $R=\Hec(n)_{A,t}$ be the corresponding Hecke algebra of type $A$.  By  work of  Dipper and James, \cite{DJ1}, Theorem 3.3(i),   if $P$ is a  (quantised) Young permutation module for $R$, then $(A,R,P)$ satisfies the standard hypotheses. This is a key example for us  and we shall give more details, in Section 5,  when we have more notation available.
\end{example}

\q We will assume until the end of this section that the triple $(A,R,P)$ satisfies the standard hypotheses.

\q Applying  Corollary 2.13  with $\Lambda=\End_R(P)$  we get the following.

\begin{proposition}  Suppose that for every residue field $F$ of $A$ the natural map $F\otimes_A \DEnd_R(P)\to \DEnd_{R_F}(P_F)$ is surjective. Then for any $A$-algebra $B$ the natural map
 $$B^\op\otimes_A \DEnd_R(P)\to \DEnd_{R_B}(P_B)$$
 is an isomorphism.

\end{proposition}

\q Let $K$ be the field of fractions of $A$.

\begin{lemma} Let $F$ be a residue field of $A$.

(i) We have $\dim_F \Ann_{R_F}(P_F)\geq \dim_K \Ann_{R_K}(P_K)$ with equality if and only if the natural map $F\otimes_A \Ann_R(P)\to \Ann_{R_F}(P_F)$ is an isomorphism.

(ii) We have $\dim_F \DEnd_{R_F}(P_F)\geq \dim_K \DEnd_{R_F}(P_F)$ with equality if and only if the natural map $F\otimes_A \DEnd_R(P) \to  \DEnd_{R_F}(P_F)$ is an isomorphism.
\end{lemma}

\begin{proof} (i) Let $h=\dim_K \Ann_{R_K}(P_K)$.   We have the isomorphism  $K\otimes_A \Ann_R(P) \to \Ann_{R_K}(P_K)$   so the torsion free rank of $\Ann_R(P)$ is $h$. Now we have the injective map $F\otimes_A \Ann_R(P)\to \Ann_{R_F}(P_F)$ so that \\
$\dim_F \Ann_{R_F}(P_F)\geq h$ with equality if and only  $F\otimes_A \Ann_R(P)\to \Ann_{R_F}(P_F)$ is an isomorphism.

(ii) Similar.
\end{proof}

\begin{lemma}   Suppose that $R_K$ is semisimple. Let $F$ be a residue field of $A$.  If $P_F$ has the double centraliser property then the natural maps
$$F\otimes_A \Ann_R(P) \to \Ann_{R_F}(P_F)  \q  \hbox{and} \q    F\otimes_A \DEnd_R(P)\to \DEnd_{R_F}(P_F)$$
are isomorphisms.
\end{lemma}

\begin{proof}  Let $n$ be the rank of $R$ as $A$-module. Since  $R_K$ is semisimple,  $P_K$ has the double centraliser property by the Jacobson Density Lemma, see e.g., \cite[XVII,\S3, Theorem 1]{Lang}.  Hence we have the short exact sequence
$$0\to \Ann_{R_K}(P_K)\to R_K\to \DEnd_{R_K}(P_K)\to 0$$
in particular 
$$n=\dim_K \Ann_{R_K}(P_K)+ \dim_K \DEnd_{R_K}(P_K)\eqno{(1)}.$$
By  assumption,  we also have the short exact sequence
$$0\to \Ann_{R_F}(P_F)\to R_F \to \DEnd_{R_F}(P_F)\to 0.$$
In particular we have
$$n=\dim_F \Ann_{R_F}(P_F)+ \dim_F \DEnd_{R_F}(P_F)\eqno{(2)}.$$
But by Lemma 3.4,  $\dim_F \Ann_{R_F}(P_F)\geq \dim_K \Ann_{R_K}(P_K)$ and \\
$\dim_F \DEnd_{R_F}(P_F)\geq \dim_K \DEnd_{R_K}(P_K)$ so that (1) and (2) imply that $\dim_F \Ann_{R_F}(P_F)=\dim_K \Ann_{R_K}(P_K)$
and  \\
$\dim_F \DEnd_{R_F}(P_F)= \dim_K \DEnd_{R_K}(P_K)$.  Another application of  Lemma 3.4  completes the proof.
\end{proof}

\q We write ${\Bbb F}_p$  for the field of prime order $p$.

\begin{proposition} Suppose that $R_K$ is semisimple. If, for each  residue field  $F$, the $R_F$-module $P_K$ has the double centraliser property, then the $R$-module $P$ has the double centraliser property.

\q In particular if $G$ is a  finite group and $X$ a finite $G$-set then the permutation module $\zed X$ for $\zed G$ has the double centraliser property provided that for each prime $p$ the ${\Bbb F}_p G$ permutation module ${\Bbb F}_p X$ has the double centraliser property.
\end{proposition}

\begin{proof}  First suppose that $A$ is a discrete valuation ring. Let $\pi$ be a generator of the maximal ideal.  We need to show that for each $\sigma\in \DEnd_R(P)$ there exists $w\in R$ satisfying $\sigma(p)=wp$ for all $p \in P$. The  map ${\tilde\sigma}\in \End_{R_K}(P_K)$, obtained from $\sigma$  by extending coefficients, lies in $\DEnd_{R_K}(P_K)$ so there is some ${\tilde w}\in R_K$ such that ${\tilde \sigma}(p)={\tilde w}p$, for all $p\in P$. Now we have $\pi^k {\tilde w}\in R$ for some positive integer $k$. So we get that, for some $w\in R$, $\pi^k\sigma(p)=wp$, for all $p\in P$.  We shall say that $\sigma$ has index $k$ if $k$ is the smallest non-negative integer such that there exists $w\in R$ such that $\pi^k\sigma(p)=wp$, for all $p\in P$.

\q Suppose that the proposition  is false and that $\sigma$ is an element of \\
$\DEnd_R(P)$ with smallest possible positive index. Suppose first that $\sigma$ has index $1$. Then, for some $w\in R$,  we have $\pi\sigma(p)=wp$, for all $p\in P$. Let ${\overline w}$ be the image of $w$ in $R_F=R/\pi R$. Then ${\overline  w}$ acts as zero on $P_F$ so that ${\overline  w}\in \Ann_{R_F}(P_F)$, which by  the hypothesis and Lemma 3.5, is   $\Ann_R(P)+\pi R$. Thus we have $w=t+\pi w'$, for some $t\in \Ann_R(P)$ and $w'\in R$.  Hence we have $\pi\sigma(p)=\pi w'p$, for all $p\in P$, and therefore $\sigma(p)=w'p$, for all $p\in P$ and $\sigma$ has index $0$, a contradiction.  So suppose $\sigma$ has index $k>1$. Then $\pi\sigma$ has index $k-1$ so, by the minimality of $k$, there exists $w\in R$ such that $\pi \sigma(p)=wp$, for all $p\in P$. But then $\sigma$ has index $1$, and this case has already been ruled out. 

\q We now  allow $A$ to be an arbitrary Dedekind domain. We consider  the natural map $\nu:R \to \DEnd_R(P)$.  Let $\M$ be a maximal ideal of $A$ then we have the localised map  
$\nu_\M:R_\M  \to \DEnd_R(P)_\M= \DEnd_{P_\M}(P_\M)$, and this we know to be an epimorphism by the case already considered,  since $A_\M$ is a discrete valuation ring.  Hence $\nu_\M$ is an epimorphism, for each maximal ideal $\M$ and so $\nu$ is an epimorphism, by \cite[Proposition 3.9]{AM}.
\end{proof}   

\begin{example} We give an example to show that base change does not always holds on annihilators of permutation modules. Let $A$ be a commutative ring.   Let $G$ be a finite group and $T$ a subgroup. We consider the $G$-set $X=G/T$.  The permutation module $A X$ may be realised as the left ideal of $A G$ generated by $\sum_{t\in T} t$.  Consider a function $f:G\to A$. The 
 element $d=\sum_{g\in G} f(g)g$ belongs to the annihilator of  the $AG$-module $AX$  if and only if $\sum_{g\in G, t\in T}  f(g)gzt=0$ for all $z\in G$.  Now the coefficient of $u\in G$ in  
 $\sum_{g\in G, t\in T}  f(g)gzt$ is $\sum_{t\in T}f(ut^{-1}z^{-1})$, so that the condition for $d$ to belong to the annihilator is $\sum_{t\in T} f(ut^{-1}z^{-1})=0$ for all $u,z\in G$, or 
$$\sum_{t\in T} f(xty^{-1})=0$$
for all $x,y\in G$.  

\q Putting $H=G\times G$, we have an action of $H$ on $G$ by $h\cdot g=xgy^{-1}$, for $h=(x,y)$, $g\in G$. Hence the condition is
$$\sum_{t\in T} f(h\cdot t)=0 \eqno{(3)}$$
for all $h\in H$. We write $C(G,T,A)$ for the $A$-module  of all functions $f:G\to A$ satisfying (3).  Now suppose that $t$ is an element of order $2$ and $T=\{1,t\}$.   We  can make a graph $\Gamma(G,T)$  out of this data,  with vertex set $G$ and edges $\{h\cdot 1, h\cdot t\}$, $h\in H$. Now $C(G,T,A)$ is the $A$-module of all $A$-valued functions  on $G$ satisfying 
$f(g_1)+f(g_2)=0$, whenever $\{g_1,g_2\}$ is an edge.

\q We consider an arbitrary finite simple graph $\Gamma=(V,E)$ and the $A$-module of functions $C(\Gamma,A)$ of all functions $f:V\to A$ such that $f(v)+f(w)=0$, whenever $\{v,w\}$ is an edge. Then $C(\Gamma,A)=\bigoplus_{i=1}^r C(\Gamma_i,A)$, where $\Gamma_i=(V_i,E_i)$, $1\leq i\leq r$, are the connected components of  $\Gamma$. 

\q Let $F$ be a field. If $F$ has characteristic $2$ then $C(\Gamma_i,F)$ consists of the constant  functions on $V_i$ , so that $\dim_F C(\Gamma_i,F)=1$ and $\dim_F C(\Gamma,F)=r$. 

\q  However,  if the characteristic of $F$ is not $2$ then $C(\Gamma_i,F)$ is $0$ if $\Gamma_i$ contains and odd cycle and otherwise is one dimensional. Hence $\dim_F C(\Gamma,F)$ is the number of components without odd cycles and in particular 
$\dim_F C(\Gamma,F)=r$  if $\Gamma$ contains no odd cycles and $\dim_F C(\Gamma,F)<r$ if $\Gamma$ contains an odd cycle. Thus if $\Gamma$ contains an odd cycle then  $\dim_F C(\Gamma,F)>\dim_K C(\Gamma,K)$, for a field $F$ of characteristic $2$ and a field $K$ of characteristic different from $2$.

\q Applying this to $\Gamma(G,T)$ we have that if this graph contains an odd cycle then   $\dim_F C(G,T,F)>\dim_K C(G,T,K)$, for a field $F$ of characteristic $2$ and a field $K$ of characteristic different from $2$. In particular if $A$ is a Dedekind domain with  characteristic $0$ field of fractions $K$ and a residue field $F$ of characteristic $2$ then the $AG$-module $AX$ does not have the base change property on annihilators.

\q One may give concrete examples as follows. Let $m$ be a positive integer divisible by $4$. Let $G$ be the symmetric group $\Sym(m)$ of degree $m$ and let $t$ be the regular involution  $(13)(24)(57)(68)\ldots$, let $x=(123)(567)(9,10,11)\ldots$ and let $y$   be the inverse of $(234)(678)(10,11,12)\ldots$.  Let $h=(x,y)$.  Then the vertices $1$, $h \cdot 1=t$ and $h\cdot t=h^2\cdot 1=x^2y^{-2}$ belong to a cycle of length $3$.  Hence $\dim_F \Ann_{FG}(FX)> \dim_K \Ann_{KG}(KX)$. 
Hence by Lemma 3.4, for $X=H/T$, the permutation module $\zed X$ does not have base change on annihilators.  Moreover, by Lemma 3.5, the permutation module  ${\Bbb F}_2X$  for ${\Bbb F}_2 G$ does not have the double centraliser property.
\end{example}

\bs\bs\bs\bs

%\newpage

%two_text_prem.tex

%\input helpfile

%\begin{document}

\centerline{\bf \LARGE   Part II: Schur algebras, Hecke algebras   }

\centerline{\bf \LARGE    and    symmetric groups }

\bs\bs

\section{The double centraliser property for  some modules for the quantised Schur algebras}

\q We fix a field $F$ and $0\neq q\in F$.  In treating quantum groups over $F$ at parameter $q$, and their representations, we suppress $q$ from the notation whenever possible, in the interests of simplicity. 

\q For a positive integer $n$ we write $G(n)$ for the corresponding quantum general linear group, as in \cite{q-Donk}.  We write $X(n)$ for $\zed^n$ and $X^+(n)$ for the set of elements $\lambda=(\lambda_1,\ldots,\lambda_n)\in X(n)$ such that $\lambda_1\geq \cdots\geq \lambda_n$. As in \cite{EGS}, we write  $\Lambda(n)$ for the set of $n$-tuples of non-negative integers and $\Lambda^+(n)$ for the set of partitions with at most $n$ parts, i.e $\Lambda^+(n)=X^+(n)\bigcap \Lambda(n)$.  By the degree $\deg(\alpha)$ of a finite string of non-negative integers $\alpha=(\alpha_1,\alpha_2,\cdots)$ we mean the sum $\alpha_1+\alpha_2+\cdots$.   For $r\geq 0$ we write $\Lambda(n,r)$ for the set of elements of $\Lambda(n)$ that have degree $r$ and write $\Lambda^+(n,r)$ for $X^+(n)\bigcap \Lambda(n,r)$.  We have the usual (dominance) partial order on $\Lambda(n)$. Thus $\alpha=(\alpha_1,\ldots,\alpha_n)\leq \beta=(\beta_1,\ldots,\beta_n)$ if $\alpha$ and $\beta$ have the same degree and $\alpha_1+\cdots+\alpha_i\leq \beta_1+\cdots+\beta_i$, for all $1\leq i\leq n$.

\q We have the Borel subgroup $B(n)$ and maximal torus $T(n)$ as in \cite{q-Donk}. Weights of $G(n)$-module and $B(n)$-modules (and $T(n)$-modules) are computed with respect to $T(n)$. 
For each $\lambda\in X^+(n)$ we have a simple $G(n)$-module $L(\lambda)$ with highest weight $\lambda$ and the modules $L(\lambda)$, $\lambda\in X^+(n)$, form a complete set of pairwise non-isomorphic simple  $G(n)$-modules.

\q  For  $\lambda\in \Lambda(n)$ we have a one dimensional $B(n)$-module $F_\lambda$ with weight $\lambda$, and the induced module $\ind_{B(n)}^{G(n)} F_\lambda$,  as in  \cite[0.21]{q-Donk}. For $\lambda\in X^+(n)$ the induced module $\ind_{B(n)}^{G(n)} F_\lambda$ is a non-zero  finite dimensional  module, which we denote $\nabla(\lambda)$. For $\lambda\in \Lambda^+(n)$ the module $\nabla(\lambda)$ is polynomial of degree $\deg(\lambda)$.

\q Recall that a finite dimensional $G(n)$-module $M$  is called a tilting module if  both  $M$ and the dual module module $M^*$ admit  filtrations with sections of the form $\nabla(\lambda)$, $\lambda\in X^+(n)$.  For $\lambda\in X^+(n)$ there is an indecomposable tilting module $T(\lambda)$, with highest weight $\lambda$, and the modules $T(\lambda)$, $\lambda\in X^+(n)$, form a complete set of indecomposable tilting modules.  For $\lambda \in \Lambda^+(n)$, the module $T(\lambda)$ is polynomial of degree $\deg(\lambda)$.

\q The category of  (left) $G(n)$-module which are polynomial of degree $r$ is naturally equivalent to the category of (left) modules for the (quantised)  Schur algebra $S(n,r)$.    The   Schur algebra $S(n,r)$ has the structure of a quasi-hereditary algebra with respect to the natural partial order on $\Lambda^+(n,r)$ and  labelling of simple modules $L(\lambda)$,   with costandard modules $\nabla(\lambda)$,  and tilting modules $T(\lambda)$,  $\lambda\in \Lambda^+(n,r)$ (for details see for example \cite[Appendix]{q-Donk}).

 \q We shall say that a subset $\sigma$ of $\Lambda^+(n,r)$ is {\it saturated}  (\resp  {\it co-saturated})  if whenever $\mu\in \sigma$, $\lambda\in \Lambda^+(n,r)$ and $\lambda\leq \mu$ (\resp  $\lambda\geq \mu$)  then $\lambda\in \sigma$.

 \q  For a finite dimensional algebra $S$ over $F$ a finite dimensional indecomposable module $U$ and a finite dimensional $S$-module $V$ we write $(V \vert U)$ for the number of components isomorphic to $U$ in a decomposition of $V$ as a direct sum of indecomposable modules.  
 
 \q For a tilting module $T$ polynomial of degree $r$ we set 
 $$\pi_s(T)=\{\lambda\in \Lambda^+(n,r) \vert (T\vert T(\lambda))\neq 0\}$$
   (the subscript $s$ indicates that the tilting module corresponds to the signed Young modules that appear in Section 5).   
 
\q  By \cite[Proposition 2.1]{LieModules} we have the following.

\begin{proposition}  If $\pi_s(T)$ is saturated then $T$, as an  $S(n,r)$-module, 
 has the double centraliser property.  
\end{proposition}

\q We write $E$ for the natural $G(n)$-module. For $r\geq 0$ we  have  the $r$th (quantised) exterior power $\tbw^rE$,  as in \cite[1.2]{q-Donk} and, for a finite string of non-negative integers  $\alpha=(\alpha_1,\alpha_2,\ldots)$ the polynomial module \\
$\tbw^\alpha E=\tbw^{\alpha_1}E\otimes \tbw^{\alpha_2}E \otimes \cdots$. We shall use that if $\beta$ is a string obtained by reordering the parts of $\alpha$ then $\tbw^\beta E$ is isomorphic to $\tbw^\alpha E$,  see \cite[Remark (i) following   3.3  (1)]{q-Donk}. 
 
 \q The transpose of a partition $\lambda$ will be denoted $\lambda'$.  For  $\lambda\in \Lambda^+(n)$   we have
$$\tbw^{\lambda'} E=T(\lambda)\oplus C\eqno{(1)}$$
where $C$ is a direct sum of tiling modules  $T(\mu)$ with $\mu$ smaller than $\lambda$ in the  dominance order.

\q  We are interested in conditions that ensure that a direct sum of modules of the form $\tbw^\alpha  E$ has the double centraliser property.  

\q Let $r\geq 0$. We write $\Par(r)$ for the set of partitions of $r$.  For a subset $\sigma$ of $\Par(r)$ we write $\sigma'$ for the set $\{\lambda'\vert\lambda \in \sigma\}$. For $\lambda\in \Par(r)$ the module  $\tbw^\lambda E$ is non-zero if and only if $\lambda\in \Lambda^+(n,r)'$. Moreover, $\tbw^\lambda E$ has highest weight $\lambda'$, for $\lambda\in \Lambda^+(n,r)'$.  Now if   $M=\oplus_{\lambda\in \Lambda(n,r)'} (\tbw^\lambda E)^{(d_\lambda)}$ and if  $\mu$ is minimal such that $d_\mu\neq 0$ then $\mu'$ is a highest weight weight of $M$ and occurs with multiplicity $d_\mu$.  It follows that if we can also write $M=\oplus_{\lambda\in \Lambda(n,r)'} (\tbw^\lambda E)^{(e_\lambda)}$ then $d_\mu=\e_\mu$. Thus $M=N\oplus (\tbw^\mu E)^{(d_\mu)}$ where  
$$N=\bigoplus_{\lambda\in \Lambda(n,r)'\backslash\{\mu\}} (\tbw^\lambda E)^{(d_\lambda)}=\bigoplus_{\lambda\in \Lambda(n,r)'\backslash\{\mu\}} (\tbw^\lambda E)^{(e_\lambda)}$$
so we get by induction that $d_\lambda=e_\lambda$, for all $\lambda\in \Lambda^+(n,r)'$.  So we have:

\begin{lemma} If the finite dimensional  $S(n,r)$-module $T$ is a direct sum of modules of the form $\tbw^\lambda E$, $\lambda\in \Lambda^+(n,r)'$,  then the multiplicities $d_\lambda$ in an expression 
$$T=\bigoplus_{\lambda\in \Lambda^+(n,r)'} ( \tbw^\lambda E)^{(d_\lambda)}$$
are uniquely determined.
\end{lemma}

\q For a module $T$ of the form $T=\bigoplus_{\lambda\in \Lambda^+(n,r)'} ( \tbw^\lambda E)^{(d_\lambda)}$ we write  $\zeta_s(T)$ for  the set of $\lambda\in \Lambda^+(n,r)'$ such that $d_\lambda\neq 0$.

\begin{lemma} Assume $r\leq n$. Let $T$ be a finite direct sum of modules of the form $\tbw^\lambda E$, $\lambda\in \Par(r)$.  Then $\zeta_s(T)'\subseteq \pi_s(T)$, with equality if $\zeta_s(T)$ is co-saturated.

\end{lemma}

\begin{proof} Let $\lambda\in \zeta_s(T)$. Then $\tbw^\lambda E$ is a direct summand of $T$ and, by (1), $T(\lambda')$ is a direct summand of $\tbw^\lambda E$, and hence of $T$, so that $\lambda'\in \pi_s(T)$.  

\q Now suppose that $\zeta_s(T)$ is co-saturated. Let  $\lambda\in \pi_s(T)$.  Then $T(\lambda)$ is a direct summand of $T$ and hence of some summand  $\tbw^\mu E$ of $T$.  Then  $\mu\in \zeta_s(T)$ and, by (1), we have $\lambda\leq \mu'$. Hence $\lambda'\geq \mu$ and,  since $\zeta_s(T)$ is co-saturated, we have $\lambda'\in \zeta_s(T)$, i.e. $\lambda\in \zeta_s(T)'$.
\end{proof}

    \q  Note that if $T$ is as in the above lemma with $\zeta_s(T)$ cosaturated  then $\pi_s(T)$ is saturated and so $T$, as a module for $S(n,r)$, has the double centraliser property by Proposition 4.1.    We can improve on this basic situation somewhat.  
    
    \q We recall, \cite[I, Section 6]{Mac}, the notion of a refinement of a partition.  If $\lambda$ and $\mu$ are partitions then we write $\lambda\cup \mu$ for the partition made from all the parts of $\lambda$ and $\mu$ by arranging them in descending order.  Let $\mu=(\mu_1,\mu_2,\ldots)$ be a partition. A partition $\lambda$ is a {\it refinement} of $\mu$ if it is possible to write $\lambda=\cup_{i\geq 1} \lambda^{(i)}$, where each $\lambda^{(i)}$ is partition of $\mu_i$. In this situation we shall also say that $\mu$ is a {\it coarsening} of $\lambda$. We shall say that $\mu$ is a {\it simple coarsening} of $\lambda$ if the parts of $\mu$ are $\lambda_1,\ldots,\lambda_{i-1},\lambda_i+\lambda_j,\lambda_{i+1},\ldots,\lambda_{j-1},\lambda_{j+1},\lambda_{j+2},\ldots$ (arranged in descending order) for some $1\leq i<j$.
    
   \q Note that $\mu$ is coarsening of $\lambda$ if and only if, for some positive integer $m$,  there exists a sequence of partitions $\mu=\mu(1),\mu(2),\ldots,\mu(m)=\lambda$, such that $\mu(k)$ is a simple coarsening of $\mu(k+1)$, for $1\leq k <m$.  Note that it $\mu$ is a coarsening of $\lambda$ then $\mu\geq \lambda$.
 
 \q For a subset $\sigma$ of $\Lambda^+(n,r)$ we write $\hatsigma$ for the coarsening of $\sigma$, i.e. the set of all $\mu\in \Lambda^+(n,r)$ such that $\mu$ is a coarsening of some $\lambda\in \sigma$.

\begin{proposition} Assume $n\geq r$. Let $T$ be a module of the form \\
$\bigoplus_{\lambda\in \Lambda^+(n,r)} (\tbw^\lambda E)^{(d_\lambda)}$, for some non-negative integers $d_\lambda$. If $\hatzeta_s(T)$ is co-saturated   then $T$ has the double centraliser property.
  \end{proposition}
  
  \begin{proof}  We argue by induction on $|\hatzeta_s(T)\backslash \zeta_s(T)|$. If this is $0$ then $\zeta_s(T)$ is co-saturated and so $T$ is a double centraliser module by the remarks preceding the lemma.
  Now suppose that $h=|\hatzeta_s(T)\backslash \zeta_s(T)|>0$ and that the result holds for all direct sums $U$ of modules of the form $\tbw^\lambda E$, with $\hatzeta_s(U)$ co-saturated and    $|\hatzeta_s(U)\backslash \zeta_s(U)|<h$.  
  
  \q We choose $\mu\in \hatzeta_s(T)\backslash \zeta_s(T)$.  Then there exists $\lambda\in \zeta_s(T)$ and partitions $\mu(1),\ldots,\mu(m)$ such that $\mu=\mu(1), \lambda=\mu(m)$ and $\mu(k)$ is a simple coarsening of $\mu(k+1)$, for $1\leq k<m$. 
  We choose $k$ maximal such that \\
  $\mu(k)\not\in \zeta_s(T)$ and replace $\mu$ by $\mu(k)$ and $\lambda$ by $\mu(k+1)$.  Thus we may assume $\mu$ is a simple coarsening of $\lambda$. Writing  $\lambda =(\lambda_1,\lambda_2,\cdots)$,   the partition  $\mu$ has parts 
  $\lambda_1,\ldots,\lambda_{i-1},\lambda_i+\lambda_j,\lambda_{i+1}\ldots, \lambda_{j-1},\lambda_{j+1},\ldots$ for some $i<j$. 
  
  \q Thus  $\tbw^{\lambda}E$ is isomorphic to $ X\otimes   \tbw^{\lambda_i}E\otimes \tbw^{\lambda_j} E \otimes Y$ and 
  $\tbw^\mu E$ to \\
  $X\otimes \tbw^{\lambda_i+\lambda_j}E\otimes Y$, for suitable tensor products of exterior powers $X$ and $Y$. We have an epimorphism $\tbw^{\lambda_i}E \otimes \tbw^{\lambda_j}E\to  \tbw^{\lambda_i+\lambda_j}E$ given by the product and hence an epimorphism  $\tbw^{\lambda}E\to \tbw^\mu E$. By \cite[Lemma 3.1]{LieModules}, $T$ has the double centraliser property if and only if $U=T\oplus \tbw^\mu E$ has the double centraliser property. However, we have $\zeta_s(U)=\zeta_s(T)\cup \{\mu\}$,  the coarsening $\hatzeta_s(U)$ of $\zeta_s(U)$ is equal to 
  $\hatzeta_s(T)$,  and 
  $$|\hatzeta_s(U)\backslash \zeta_s(U)|=|\hatzeta_s(T)\backslash \zeta_s(T)|-1=h-1.$$
    Hence $U$ has the double centraliser property and so too  does $T$.
  \end{proof}

 \begin{example}
  We note the following special case.   Suppose that $r=a+b\leq n$.  We consider the coarsening  $\sigma$ of $\sigma=\{(a,1^b)\}$.  We claim that 
 $$\hatsigma=\{\lambda=(\lambda_1,\lambda_2,\ldots) \in \Lambda^+(n,r) \vert \lambda_1\geq a\}.$$
  Let $\lambda\in \hatsigma$. Then $\lambda\geq (a,1^b)$ so $\lambda_1\geq 1$.    If $\lambda_1=a$ then $(\lambda_2,\lambda_3,\ldots)$ is a partition of $b$.  Clearly any partition of $b$ is a coarsening of $1^b$ and so $\lambda=(a,\lambda_2,\lambda_3,\ldots)$ is a coarsening of $(a,1^b)$. If $\lambda_1>a$ then for some $i$ we have $\lambda_i>a$, $\lambda_{i+1} \leq a$.  We may assume inductively that $\mu=(\lambda_1,\ldots,\lambda_{i-1},\lambda_i-1,\lambda_{i+1},\ldots)$ is a coarsening of $(a,1^{b-1})$. Hence $\mu \cup \{1\}$ is a coarsening of $(a,1^b)$ and $\lambda$ is a coarsening of $\mu\cup  \{1\}$  and hence of $(a,1^b)$.
  \end{example}

    \begin{corollary}    Let $1  \leq r\leq  n$ and let $a_1,\ldots,a_m$ be positive integers and $ b_1,\ldots,b_m$ be non-negarive integers such that $a_i+b_i=r$ for $1\leq i\leq m$. Then the tilting module 
    $$T=\bigoplus_{i=1}^m \tbw^{a_i}E \otimes E^{\otimes b_i}$$
    for $S(n,r)$,   has the double centraliser property.
  \end{corollary}
  
  \begin{proof}   Let  $\sigma_i=\{(a_i,1^{b_i})\}$. Then  $\hatsigma_i=\{\lambda\in \Par(r) \vert \lambda_1\geq a_i\}$ so that
  \begin{align*}\hatzeta_s(T)&=\bigcup_{i=1}^m \hatsigma_i=\bigcup_{i=1}^m \{\lambda\in \Par(r) \vert \lambda_1\geq a_i\}\cr
  &=\{\lambda\in \Par(r) \vert \lambda_1\geq a\}
  \end{align*}
  where $a$ is the minimum of $a_1,\ldots,a_m$, by the example,  and this is a co-saturated set so $T$ has the double centraliser property, by Proposition 4.4.
   \end{proof}

 \begin{remark} Let $\sigma$ be a finite set of finite  strings of non-negative integers of fixed degree $r$. It would be interesting to have a precise combinatorial criterion on $\sigma$ for the  $S(n,r)$-module $\bigoplus_{\lambda\in \sigma}  \tbw^\lambda E$, to have the double centraliser property.  Some examples in which this fails are given in \cite[Section 6]{LieModules}. We shall also see, in Section 7 below,  that it fails for the $S(4,4)$-module $\tbw^2 E \otimes \tbw^2E$ in characteristic $2$.
 \end{remark}

% \newpage

 \section{The double centraliser property for  some modules for the Hecke algebra}

\q We will be working with quantised   Young permutation  modules over a Hecke algebra. But first we record the following in the general context.  It will be used to twist with the sign representation. We leave the details to the reader.

  \begin{remark}  Let $R$ be a finite dimensional algebra over a field $F$ and let $\alpha$ be an algebra automorphism of $R$.   For a finite dimensional module $V$, affording representation $\rho:R\to \End_F(V)$,  we denote by $V^\alpha$ the corresponding twisted module, i.e. the vector space $V$ regarded as an $R$-module via the representation $\rho\circ \alpha: R\to \End_F(V)$. The annihilator $\Ann_R{V^\alpha}$ is equal to $\alpha^{-1}(\Ann_R(V))$.  Further $\DEnd_R(V^\alpha)=\DEnd_R(V)$ and   $V$ has the double centraliser property if and only if $V^\alpha$ has.
\end{remark}

\q We now make a general remark on   double centralisers.  Let $S$ be a finite dimensional algebra over a field $F$. Let $e$ be a non-zero idempotent in $S$. Thus we have the algebra $eSe$.  
The argument given below is  taken from the proof of \cite[Proposition 7.1]{LieModules}  (but since the context is slightly different we give it again).

\begin{proposition}  Suppose $T$ is a finite dimensional module such that the natural map $\End_S(T)\to \End_{eSe}(eT)$ is an isomorphism. If $T$ is a double centraliser module for $S$ then $eT$ is a double centraliser module for $eSe$.
\end{proposition}

\begin{proof}  Let $\phi\in \DEnd_{eSe}(eT)$. We define $\hatphi\in \End_F(T)$ by 
$$\hatphi(x)=\begin{cases} \phi(x), & \hbox{ if } x\in eT;\cr
0, & \hbox{ if } x\in (1-e)T.
\end{cases}$$

Let $g\in \End_S(T)$ and let $f=g\vert_{eT}\in \End_{eSe}(eT)$.  Then for $x\in eT$, we have
$$\hatphi\circ g(x)=\phi(f(x))=f(\phi(x))=g\circ \hatphi(x)$$
and, for $x\in (1-e)T$, we have 
$$\hatphi\circ g(x)=\hatphi(g(1-e)x)=\hatphi((1-e)g(x))=0$$
and $g\hatphi(x)=0$ (since $x\in (1-e)T$). 
Hence we have  $\hatphi \circ g=g\circ\hatphi$ and therefore $\hatphi\in \DEnd_S(T)$. 

\q  Since $T$ has the double centraliser property for some $s\in S$  we have $\hatphi(x)=\rho(s)(x)$, for all $x\in T$,  where $\rho:S\to \End_F(T)$ is the representation afforded by $T$. Now for $x\in eT$ we have $\phi(x)=e\phi(x)=e\hatphi(x)=esx=esex$, so we have $\phi=\rho_e(ese)$, where $\rho_e:eSe\to \End_F(eT)$ is the representation afforded by the $eSe$-module $eT$. Hence $eT$ has the double centraliser property.
\end{proof}

  \q We shall apply the results of Section 4   to the representation theory of Hecke algebras.   We fix a positive integer $n$.  As usual we take $A$ to be a Dedekind domain or a field. Let $q$ be a unit in $A$. We have the Hecke algebra $\Hec(n)_{A,q}$, with the usual generators $T_i=T_{i,A,q}$, $1\leq i < n$, and basis $T_w=T_{w,A,q}$, $w\in \Sym(n)$.

\q   We have the involutory  automorphism $\sharp=\sharp_{A,q}:\Hec(n)_{A,q}\to \Hec(n)_{A,q}$ given on generators by $\sharp(T_i)=-T_i +(q-1)1$, $1\leq i<n$. For a $\Hec(n)_{A,q}$-module $U$ affording the representation $\rho:\Hec(n)_{A,q}\to \End_A(U)$,  write $U^\sharp$ for the $\Hec(n)_{A,q}$ module on vector space $U$ with action $\rho\circ \sharp$.  For $\alpha=(\alpha_1,\alpha_2,\ldots)\in \Lambda(n,n)$ we have the corresponding Young subgroup $\Sym(\alpha)\cong \Sym(\alpha_1)\times \Sym(\alpha_2)\cdots$ of $\Sym(n)$ and corresponding elements  
  $$x(\alpha)=x(\alpha)_{A,q}=\sum_{w\in \Sym(\alpha)} T_{w,A,q}$$
  and
  $$y(\alpha)=\sum_{w\in \Sym(\alpha)} (-q)^{N-l(w)} T_{w,A,q}$$
  
  of $\Hec(n)_{A,q}$, where $N={n\choose 2}$ and where $l(w)$ denotes the length of an element $w$ of $\Sym(n)$.

  \q We define a subset  $J(\alpha)$ of $\{1,\ldots,n-1\}$ as follows.   We define $J_1(\alpha)$ to be the set of $i$ such that $1\leq i\leq \alpha_1$ and 
  , for $r>1$ define $J_r(\alpha)$ to be the set of $i$  such that $\alpha_1+\cdots+\alpha_{r-1} < i < \alpha_1+\cdots+\alpha_r$.  We put $J(\alpha)=\bigcup_{r=1}^n J_r(\alpha)$. We have the Young subalgebra $\Hec(\alpha)_{A,q}$, generated by all $T_i$, with $i\in J(\alpha)$.  We write simply  $A$ for the regular $A$-module regarded as a $\Hec(n)_{A,q}$-module on which each $T_i$ acts as multiplication by $q$, and write $A_s$ for the  regular $A$-module regarded as a $\Hec(n)_{A,q}$-module on which each $T_i$ acts as multiplication by $-1$.   For a partition $\lambda$ of $n$  we have the  \lq\lq permutation module" 
  $M(\lambda)_{A,q}=\Hec(n)_{A,q}\otimes_{\Hec(\lambda)_{A,q}}  A$ and the  \lq\lq signed permutation module"  $M_s(\lambda)_{A,q}=\Hec(n)_{A,q}\otimes_{\Hec(\lambda)_{A,q}} A_s$. The module $M(\lambda)_{A,q}$ may be realised as the left ideal of $\Hec(n)_{A,q}$ generated by $x(\lambda)_{A,q}$ and the module $M_s(\lambda)_{A,q}$ may be realised as the left ideal of $\Hec(n)_{A,q}$ generated by $y(\alpha)_{A,q}$. 
  
  \q By a  (quantised) Young permutation module for $\Hec(n)_{A,q}$ we mean a finite direct sum of copies of modules of the form $M(\lambda)_{A,q}$, $\lambda\in \Lambda^+(n,n)$  and by a (quantised) signed Young permutation module $\Hec(n)_{A,q}$ we mean a finite direct sum of copies of modules of the form $M_s(\lambda)_{A,q}$, $\lambda\in \Lambda^+(n,n)$. 
  
  \q  Let $\lambda$ be a partition of $n$.   A left coset $w\Sym(\lambda)$ of $w\in \Sym(n)$ contains a unique element of minimal length (see \cite[Lemma 1.1]{DJ1}). We write ${\mathcal D}_\lambda$ for the set of minimal length elements of left cosets of $\Sym(\lambda)$ in $\Sym(n)$.   Then  $M(\lambda)_{A,q}$ has $A$-basis $T_{w,A,q}x(\lambda)_{A,q}$, $x\in {\mathcal D}_\lambda$ and $M_s(\lambda)_{A,q}$ has $A$-basis $T_{w,A,q}y(\lambda)_{A,q}$, $x\in {\mathcal D}_\lambda$.

  \q We shall need a result of Dipper and James,  \cite[Theorem 3.4]{DJ1}, on the invariance of spaces of homomorphisms between  Young permutation modules.   Suppose that  $\phi:A\to \barA$ is a  ring homomorphism, where $\barA$ is also a Dedekind domain or a field. Let $\barq=\phi(q)$. Then we have the induced $A$-algebra homomorphism $\Hec(n)_{A,q}\to \Hec(n)_{\barA,\barq}$. We identify $\barA\otimes_A \Hec(n)_{A,q}$ with $\Hec(n)_{\barA,\barq}$ via the induced isomorphism $\barA\otimes_A \Hec(n)_{A,q}\to \Hec(n)_{\barA,\barq}$. Thus, if $M$ is a $\Hec(n)_{A,q}$-module then we have the $\Hec(n)_{\barA,\barq}$-module $M_{\barA}=\barA\otimes_A M$ obtained by base change. In particular we have $\barA \otimes_A M(\lambda)_{A,q}=M(\lambda)_{\barA,\barq}$, for $\lambda\in \Par(n)$.
  
  \begin{proposition} Let $M,N$ be Young permutation modules (\resp signed Young permutation modules) for $\Hec(n)_{A,q}$. Then the natural map
  $$\barA\otimes_A \Hom_{\Hec(n)_{A,q}}(M,N)\to \Hom_{\Hec(n)_{\barA,\barq}}(M_{\barA},N_{\barA})$$
  is an isomorphism.
   \end{proposition}
   
   \begin{proof} We assume that $M,N$ are Young modules (the other case is similar). It is enough to consider the case in which $M=M(\lambda)_{A,q}$, $N=M(\mu)_{A,q}$, for some $\lambda,\mu\in \Par(n)$, and  the result holds in this case by the explicit description of spaces of homomorphisms given by Dipper and James in \cite[Theorem 3.4]{DJ1}. Thus by Corollary 2.13 we have the following.
   \end{proof}

   \begin{corollary}   Let $P$ be a   Young permutation module or a signed permutation module.    For any $A$-algebra $B$ the natural map
$$\Psi_B: B^\op\otimes_A \End_R(P)\to \End_{R_B}(P_B)$$
  is an isomorphism. 
\end{corollary}

  \q We now focus on  the case in which the coefficient ring is a field $F$ and write simply $\Hec(n)$ for $\Hec(n)_{F,q}$, write   $M(\lambda)$ for $M(\lambda)_{F,q}$ and write $M_s(\lambda)$ for $M_s(\lambda)_{F,q}$, $\lambda\in \Par(n)$. 
  
    \q We shall need  the symmetric powers of the natural module for the corresponding  quantum general linear group  $G(n)$.    For $a\geq 0$ we write $S^aE$ for the $a$th (quantised) symmetric power of $E$ and for a finite string $\alpha=(\alpha_1,\alpha_2,\ldots)$ of non-negative integers, write $S^\alpha E$ for the tensor product\\
   $S^{\alpha_1}E\otimes S^{\alpha_2}E \otimes \cdots$, as in \cite[Section 2.1]{q-Donk}.  For $\lambda\in \Lambda^+(n)$ we write $I(\lambda)$ for the injective envelope of $L(\lambda)$ in the polynomial category.  If $\lambda$ has degree $r$ then $I(\lambda)$ is polynomial of degree $r$.  If $r\leq n$ then for $\lambda\in \Lambda^+(n,r)$ the module $S^\lambda E$ is injective, in the polynomial category. 
   
   \q   We identify $\Hec(n)$ with  the algebra $eS(n,n)e$, for a certain idempotent $e\in S(n,n)$, as in \cite[Section 2.1]{q-Donk}   (or  \cite[Chapter 6]{EGS} in the classical case $q=1$).  We have the Schur functor $f:\mod(S(n,n))\to \mod(\Hec(n))$, as in \cite[Chapter 6]{EGS}. Then we    $M(\lambda)=f S^\lambda E$ and $M_s(\lambda)=f \tbw^\lambda E$, $\lambda\in \Par(n)$. We 
  have the module   $Y(\lambda)= fI(\lambda)$ and $\Sp(\lambda)=f\nabla(\lambda)$, the quantised  Young  module  and Specht module corresponding to $\lambda$, for $\lambda\in \Par(n)$.  We also have  signed versions $M_s(\lambda)=f \tbw^\lambda E$, $Y_s(\lambda)=f T(\lambda)$, of the Young permutation module and Young module, for $\lambda\in \Par(\lambda)$. 
  (The modules $Y_s(\lambda)$, $\lambda\in \Par(n)$, are, in the classical case,  special cases of the modules called signed Young modules in \cite{DonkSymmExt}. However the modules $Y_s(\lambda)$ are sufficient for our present purpose.)

  \q The modules $Y(\lambda)$ (\resp $Y_s(\lambda)$), for $\lambda\in \Par(n)$, are indecomposable and pairwise non-isomorphic.
  
  \q By a quantised Young module (\resp  signed quantised  Young module) we mean a finite direct sum of modules $Y(\lambda)$ (\resp $Y_s(\lambda)$),  $\lambda\in \Par(n)$.    By a  quantised  Young permutation module (\resp signed  quantised  Young permutation module) we mean a finite direct sum modules $M(\lambda)$  
   (\resp $M_s(\lambda)$), $\lambda\in \Par(n)$.  
   
   \q For a Young module $U$  we set $\pi(U)=\{\lambda\in \Par(n) \vert (U \vert Y(\lambda))\neq 0\}$. For a signed Young module $V$ we set 
$\pi_s(V)=\{\lambda\in \Par(n) \vert (V \vert Y_s(\lambda))\neq 0\}$. 

\begin{remark} We shall need to use the fact that if $T,U$ are tilting modules for $S(n,n)$ then the Schur functor induces an isomorphism \\
$\Hom_{S(n,n)}(T,U)\to  \Hom_{\Hec(n)}(fT,fU)$. It suffices to observe this in the case in which $U=T(\lambda)$, for some $\lambda\in \Par(n)$.  Moreover, since $T(\lambda)$ is a direct summand of $\tbw^{\lambda'} E$, it is enough to note this in the case $U=\tbw^{\lambda'}E$, and this is true by \cite[2.1,(16)((iii)]{q-Donk}.

\end{remark}

\q Our main interest is in the study of double centralisers and annihilator ideals of Young permutation modules and signed Young permutation modules. But for completeness we also consider   the situation for Young modules and signed Young modules.

  \q  By \cite[4.4, (10)]{q-Donk} and \cite[2.1,(20)]{q-Donk} we have $M(\lambda)^\sharp=M_s(\lambda)$, for $\lambda\in \Par(n)$.   A similar relationship exists  between, the Young modules and their signed analogues. This is not recorded in \cite{q-Donk}. However the result is
  $$Y_s(\lambda)=Y(\lambda')^\sharp \eqno{(1)}$$
  for $\lambda\in \Par(n)$, and this follows from the argument of \cite[(3.6) Lemma (ii)]{DTilt}, applied in the set-up of \cite[Section 4.4]{q-Donk}.
  
     \q We now consider  the double centraliser property for (quantised) Young permutation modules  and their signed versions.  For $\lambda\in \Par(n)$ we may write  $M(\lambda)=\bigoplus_{\mu\in \Par(n)} Y(\mu)^{(a_{\lambda\mu})}$  where $a_{\lambda\mu}=(M(\lambda)\vert  Y(\mu))$.  Moreover, we have $a_{\lambda\lambda}=1$ and $\lambda\geq\mu$ whenever $a_{\lambda\mu}\neq 0$.  Thus the matrix $(a_{\lambda\mu})_{\lambda,\mu\in \Par(n)}$ is invertible and this gives  the following for Young permutation modules in the case in which $A$ is a field.  The result then  follows in general by considering $M_F$, where $F$ is a residue field of $A$. The version for signed Young permutation modules may be obtained in a similar fashion.

  \begin{lemma}  Let $A$ be a Dedekind domain or a field and $q$ a unit in $A$.  If $M$ is a Young permutation module (\resp signed permutation module)  then the multiplicities $d_\lambda$ (\resp $e_\lambda$) , $\lambda\in \Par(n)$, in an expression $M=\bigoplus_{\lambda} M(\lambda)_{A,q}^{(d_\lambda)}$ (\resp   $M=\bigoplus_{\lambda} M_s(\lambda)_{A,q}^{(e_\lambda)}$) 
  are uniquely determined. 
  
 \end{lemma}

 \q Let $A$ be a Dedekind domain or a field and $q$ a unit in $A$.   For a Young permutation module $M=\bigoplus_{\lambda} M(\lambda)_{A,q}^{(d_\lambda)}$ we set  $\zeta(M)=\{\lambda\in \Par(n) \vert d_\lambda\neq 0\}$.  For a signed Young permutation module $N=\bigoplus_{\lambda} M_s(\lambda)_{A,q}^{(e_\lambda)}$  we set $\zeta_s(N)=\{\lambda\in  \Par(n) \vert e_\lambda\neq 0\}$.  

\q Note that this notation is consistent with the earlier usage in the sense that if $T$ is tilting module for $S(n,n)$ then $\pi_s(T)=\pi_s(fT)$ and if $T$ is a tilting module of exterior type, for $S(n,n)$ then $\zeta_s(T)=\zeta_s(fT)$.

\begin{proposition} Let $V$ be a Young permutation module (\resp signed Young permutation module)  and suppose that $\hatzeta(V)$ (\resp $\hatzeta_s(V)$) is co-saturated. Then $V$ has the double centraliser property.
\end{proposition}

\begin{proof} We consider the signed version. (The other case follows from Remark 5.1).  Let $\sigma=\zeta_s(V)$. We may assume that $V=\bigoplus_{\lambda\in \sigma} M_s(\lambda)$, by \cite[(1.5)]{LieModules}.  Consider the $S(n,n)$-module 
$$T=\bigoplus_{\lambda\in \sigma} \tbw^{\lambda} E.$$
Since $\hatsigma$ is co-saturated the $S(n.n)$-module $T$ has the double centraliser property, by Proposition 4.4.  Applying the Schur functor we obtain that $V$ has the double centraliser property, by Proposition 5.2. 
\end{proof}

\begin{corollary} Let $V$ be a Young module (\resp signed Young module) and suppose that $\pi(V)$ (\resp $\pi_s(V)$) is co-saturated. Then $V$ has the double centraliser property.
\end{corollary}

\begin{proof} We prove the version for Young modules.  By \cite[(1.5)]{LieModules} we may assume that $V=\bigoplus_{\lambda\in \pi(V)} Y(\lambda)$. Consider the Young permutation module $U=\bigoplus_{\lambda\in \pi(V)} M(\lambda)$.  Since $Y(\lambda)$ is a direct summand of $M(\lambda)$, for $\lambda\in \Par(n)$, the module $V$ is a direct summand of $U$. Moreover, an indecomposable summand of $M(\lambda)$ has the form $Y(\mu)$, for some $\mu\geq \lambda$ and, since $\pi(V)$ is co-saturated, each summand of $U$ has the form $Y(\mu)$ for some $\mu\in \pi(V)$.  Hence $V$ and $U$ have the same indecomposable summands (up to isomorphism) and hence $V$ has the double centraliser property if and only if $U$ has, by  \cite[(1.5)]{LieModules}. But $U$ has the double centraliser property by the Proposition, so we are done.
\end{proof}

 \q We now turn  our attention to the integral case.   Let $A$ be a Dedekind domain or field and let $q$ be a unit in $A$.  Let $K$ be the field of fractions of $A$. 
 
 \q Recall that the  Hecke algebra $\Hec(n)_{K,q}$ is semisimple if either $q$ is not a root of unity or if $q=1$ and $K$ has characteristic $0$.

    \q For $\lambda\in \Par(n)$ we write $\dim(\lambda)$ for the dimension of the Specht module $\Sp(\lambda)$. Thus $\dim(\lambda)$ is the dimension of the ordinary irreducible module corresponding to $\lambda$  for the symmetric group and the for the Hecke algebra at non-zero parameter which is not a root of unity.

\begin{proposition} Let $A$ be a Dedekind domain or a field and let $q$ be a unit in $A$. Let  $P$ be  a  $\Hec(n)_{A,q}$-module.  Suppose that either 
$P$ is a Young permutation module and that $\hatzeta(P)$  is co-saturated or 
$P$ is a signed Young permutation and $\hatzeta_s(P)$ is co-saturated. We write $\sigma$ for $\zeta(P)$ if $P$ is a Young permutation module and for $\zeta_s(P)$ if $P$ is a signed Young permutation module.

(i) If $A=F$ is a field then     
 $$\dim_F \Ann_{\Hec(n)_{F,q}}(P)=n!-\sum_{\lambda\in \hatsigma} \dim(\lambda)^2$$
 and 
 $$\dim_F \DEnd_{\Hec(n)_{F,q}}(P)=\sum_{\lambda\in \hatsigma} \dim(\lambda)^2.$$

 (ii) If $F$ is a residue field of $A$  and $\barq$ is the image of $q$ in $F$ then   the natural maps 
$$F\otimes_A \Ann_{\Hec(n)_{A,q}}(P)\to \Ann_{\Hec(n)_{F,\barq}}(P_F)$$
 and 
 $$F\otimes_A \DEnd_{\Hec(n)_{A,q}}(P)\to \DEnd_{\Hec(n)_{F,\barq}}(P_F)$$
 are isomorphisms.
 \end{proposition}

 \begin{proof}  We treat the case of Young permutation modules. The signed case is similar.

 (i) Consider the Hecke algebra $\Hec(n)_{F[t,t^{-1}],t}$.  We write \\
 $P=\bigoplus_{\lambda\in \Par(n)} M(\lambda)_{F,q}^{(d_\lambda)}$, for certain non-negative integers $d_\lambda$. We define $\tilde P$ to be the  $\Hec(n)_{F[t,t^{-1}],t}$-module  
 $\bigoplus_{\lambda\in \Par(n)} M(\lambda)_{F[t,t^{-1}],q}^{(d_\lambda)}.$
 
    \q We regard $F$ as a reside field of  the ring of Laurent polynomials $F[t,t^{-1}]$  with corresponding ideal generated by $t-q$.  Thus we have ${\tilde P}_F=P$.   The field of fractions of $F[t,t^{-1}]$ is the field of rational functions $F(t)$ and $\Hec(n)_{F(t),t}$ is semisimple.  Moreover, $P$ has the double centraliser property, by Proposition 5.7, so that, by Lemma 3.5, the natural map
 $$F\otimes_{F[t,t^{-1}]} \Ann_{\Hec(n)_{F[t,t^{-1}],t}}(\tilde P)\to \Ann_{\Hec(n)_{F,q}}(P)\eqno{(1)}$$
 is an isomorphism.  Now the irreducible constituents of the semisimple $\Hec(n)_{F(t),t}$-module ${\tilde P}_{F(t)}$ are the Specht modules $\Sp(\lambda)_{F(t),t}$, with $\lambda\in \hatsigma$.  Hence 
 \begin{align*}\dim_{F(t)} \Ann_{\Hec(n)_{F(t),t}}({\tilde P}_{F(t)})&=\dim_{F(t)} \Hec(n)_{F(t),t} - \sum_{\lambda\in \hatsigma} \dim(\lambda)^2\cr
& =n!- \sum_{\lambda\in \hatsigma} \dim(\lambda)^2.
 \end{align*}
 Hence the rank of $\Ann_{\Hec(n)_{F[t,t^{-1}]}}(\tilde P)$ is $n!-\sum_{\lambda\in \hatsigma} \dim(\lambda)^2$ and by (1) this is also the $F$-dimension of $\Ann_{\Hec(n)_{F,q}}(P)$, giving the first assertion. The second assertion follows from the fact that we have the short exact sequence
 $$0\to \Ann_{\Hec(n)_{F,q}}(P)\to \Hec(n)_{F,q}\to \DEnd_{\Hec(n)_{F,q}}(P)\to 0.$$
 
 (ii)    Let $K$ be the field of fractions of $A$. By part (i) we have 
 $$\dim_K \Ann_{\Hec(n)_{K,q}}(P_K)=\dim_F \Ann_{\Hec(n)_{F,\barq}}(P_F)=n!-\sum_{\lambda\in \hatsigma} \dim(\lambda)^2$$
 and
 $$\dim_K \DEnd_{\Hec(n)_{K,q}}(P_K)=\dim_F \DEnd_{\Hec(n)_{F,\barq}}(P_F)=\sum_{\lambda\in \hatsigma} \dim(\lambda)^2.$$
 Hence the torsion  free  rank of   $\Ann_{\Hec(n)_{A,q}}(P)$, as an $A$-module,  is equal to  $\dim_F \Ann_{\Hec(n)_{F,\barq}}(P_F)$ and the torsion free rank of $ \DEnd_{\Hec(n)_{A,q}}(P)$ is equal to $\dim_F \DEnd_{\Hec(n)_{F,\barq}}(P_F)$. Hence we have
 $$\dim_F F\otimes_A \Ann_{\Hec(n)_{A,q}}(P)= \dim_F \Ann_{\Hec(n)_{F,\barq}}(P_F)$$
 and
 $$\dim_F F\otimes_A  \DEnd_{\Hec(n)_{A,q}}(P)=\dim_F \DEnd_{\Hec(n)_{F,\barq}}(P_F)$$
 so the injective maps 
  $$F\otimes_A \Ann_{\Hec(n)_{A,q}}(P) \to \Ann_{\Hec(n)_{F,\barq}}(P_F)$$
and
 $$F\otimes_A  \DEnd_{\Hec(n)_{A,q}}(P)\to   \DEnd_{\Hec(n)_{F,\barq}}(P_F)$$ 
 are isomorphisms.
  \end{proof}

     \q Now from Proposition 3.3 we get the following.

 \begin{corollary}   Let $A$ be a Dedekind domain or a field and let $q$ be a unit in $A$.  Let  $P$ be  a  
 $\Hec(n)_{A,q}$-module.  Suppose that either 
$P$ is a Young permutation module and that $\hatzeta(P)$  is co-saturated or 
$P$ is a signed Young permutation and $\hatzeta_s(P)$ is co-saturated. Then $P$ has the double centraliser property and for every $A$-algebra $B$ the natural map 
$$B^\op\otimes_A \DEnd_{\Hec(n)_{A,q}}(P)\to \DEnd_{B\otimes_A \Hec(n)_{A,q}}(B\otimes_A P)$$
is an isomorphism.
 \end{corollary}

%\newpage

\section{Example: Partition Algebras}

\q In the examples we shall only treat the case of Young modules and leave the analogous case of signed Young modules to the interested reader.

\q We fix a positive integer $n$.  Let $A$ be a Dedekind domain or a field  and $q$ a unit in $A$. By a hook permutation module for $\Hec(n)_{A,q}$ we mean a non-zero  finite direct sum of modules of form $M(a,1^b)_{A,q}$, with $a+b=n$.  

\begin{definition} We define the index  $\idx(H)$   of a  hook permutation module  $H$ for $\Hec(n)_{A,q}$  to be  the smallest $a\geq 1$ such that $H$, written as a direct sum of modules of the form $M(\lambda)_{A,q}$, with $\lambda\in \Par(n)$, the module  $M(a,1^b)$ appears as a direct summand of $H$. 
\end{definition} 

\begin{definition}  For $1\leq m\leq n$ we set 
$$N_{n,m}=\sum_\lambda \dim(\lambda)^2$$
where the sum is over all partitions $\lambda$  of $n$ with first part at least $m$.

\end{definition}

We note that if $H$ is a  hook permutation module then $\hatzeta(H)$ is co-saturated and indeed $\hatzeta(H)=\{\lambda=(\lambda_1,\lambda_2,\ldots)\in \Par(n) \vert \lambda_1\geq \idx(H)\}$, by Example 4.5.

\q From Proposition 5.9, Corollary 5.10 and Corollary 2.13 we have the following.

\begin{lemma} (i) Let $F$ be a field and $0\neq q\in F$.  Let $H$ be a hook permutation module for $\Hec(n)_{F,q}$.   Then $H$ has the double centraliser property. Furthermore we have

$$\dim_F \Ann_{\Hec(n)_{F,q}}(H)=n!-N_{n,m}$$
and
$$\dim_F \DEnd_{\Hec(n)_{F,q}}(H)=N_{n,m}$$
where $m=\idx(H)$.

(ii)  Let $A$ be a Dedekind domain or a field and let $q$ be a unit in $A$.   Let $H$ be a hook permutation module for $\Hec(n)_{A,q}$.   Let $B$ be an $A$-algebra. Then $H_B$ is a double centraliser module for $B\otimes_A \Hec(n)_{A,q}$ and indeed the natural maps
$$B \otimes_A \Ann_{\Hec(n)_{A,q}}(H)\to \Ann_{B\otimes_A \Hec(n)_{A,q}}(H_B)$$
and
$$B^\op \otimes_A \DEnd_{\Hec(n)_{A,q}}(H)\to \DEnd_{B\otimes_A \Hec(n)_{A,q}}(H_B)$$
are isomorphisms.
\end{lemma}

\q We specialise to the symmetric group case.  Let $X$ be a  $\Sym(n)$-set. We say that $X$ is a Young $\Sym(n)$-set if it is finite and each point stabilizer is a Young subgroup. For a Young $\Sym(n)$-set $X$ we write $\zeta(X)$ for the set of $\lambda\in \Par(n)$ such that $\Sym(\lambda)$ is a point stabiliser. We shall say that $X$ is hook $\Sym(n)$-set if $\zeta(X)$ consists of partitions of the form $(a,1^b)$ (with $a+b=n)$.  We define the index $\idx(X)$ of a hook $\Sym(n)$-set to be the smallest integer  $a\geq 1$ such that the Young subgroup $\Sym(a,1^b)$ is a point stabilizer.

\q We shall give an  application to partition algebras.    Let $V$ be a free abelian group of rank $n$ on basis $v_1,\ldots,v_n$. For a commutative ring $B$ we have the  $B$-module $V_B=B\otimes_\zed V$ and this is a free $B$-module with basis $1\otimes v_1,\ldots, 1\otimes v_n$.  The symmetric group $\Sym(n)$ acts on $V_B$, by permuting the basis elements, and so on the $r$-fold  tensor product  
$V_B^{\otimes r}=V_B\otimes_B \cdots \otimes_B V_B$.  Thus $V_B^{\otimes r}$ is a permutation module for $B\Sym(n)$, and hence also for $B\Sym(m)$, for $1\leq m\leq n$  (where $\Sym(m)$ is identified with the subgroup of $\Sym(n)$ consisting of the elements that fix $m+1,\ldots,n$).    Let $I(n,r)$ be the set of maps from $\{1,\ldots,r\}$ to $\{1,\ldots,n\}$. We regard $i$ as an $r$-tuple $(i_1,\ldots,i_r)$ with entries in $\{1,\ldots,n\}$ (where $i_k=i(k)$, $1\leq k\leq r$).  The elements $v_i$, $i\in I(n,r)$, form a $\zed$-basis for $V^{\otimes r}$ and  $w v_i=v_{w\circ i}$, for $w\in \Sym(n)$, $i\in I(n,r)$.  We identify $V_B^{\otimes r}$ with $V_B^{\otimes}$ in the obvious way.  Thus $V_B^{\otimes r}$ is the $B\Sym(n)$ permutation module on the $\Sym(n)$-set $I(n,r)$.

\q Now the point stabilizer of $i\in  I(n,r)$ in $\Sym(n)$, is the subgroup consisting of elements   constant on the image of $i$. Hence the stabilizer in $\Sym(n)$ is conjugate to $\Sym(a,1^b)$, where $b=|{\rm Im}(i)|$. Thus 
$$\zeta(I(n,r))=\{(a,1^b) \in \Par(n) \vert 1\leq b\leq r\}.$$ 
  Now suppose that $m\leq n$. Then $w\in \Sym(m)$ fixes $i$ if  an only if it fixes all  elements of ${\rm Im}(i)\bigcap \{1,\ldots,m\}$.  Thus, for $m\leq  n$,  $I(n,r)$ is a hook $\Sym(m)$-set and:

$$\zeta(I(n,r) |_{\Sym(m)})=\begin{cases} \{(a,1^b)\in \Par(m) \vert 1\leq b\leq r\},  &   \hbox{ if }   m=n;\cr 
 \{(a,1^b)\in \Par(m) \vert 0\leq b\leq r\},  & \hbox{ if } m<n
\end{cases}\eqno{(1)}.$$
It follows that 
$$\idx(I(n,r)|_{\Sym(m)})= m- \min(r,m)\eqno{(2).}$$
(We are writing $\min(d,e)$ for the minimum of non-negative integers $d$ and $e$.)

\q Hence we also have the following.

\begin{proposition}  Let $m\leq n$ and $r\geq 1$.  Then the $B\Sym(m)$-module  $V^{\otimes r}$ has the double centraliser property moreover the natural maps 
$$B\otimes_\zed \Ann_{\zed \Sym(m)} (V^{\otimes r} ) \to \Ann_{B\Sym(m)} (V_B^{\otimes r})$$
and 
$$B\otimes_\zed \DEnd_{\zed \Sym(m)}(V_\zed^{\otimes r})   \to \DEnd_{B\Sym(m)}(V_B^{\otimes r})$$
are isomorphisms.    As a $B$-module  $\Ann_{B\Sym(m)}(V_B^{\otimes r})$ is free of rank  $m!- N_{m,d}$  and $\DEnd_B(V_B^{\otimes r})$ is free of rank $M_{d,m}$, where $d=\min(m,r)$.

\end{proposition}

\q Recall that, for an element $\de$ of $B$ we have the  partition algebra $P_r(\de)_B$ over $B$. One may find a detailed account of the construction and properties of $P_r(\de)_B$ in for example the papers by  Paul P. Martin, \cite{PPM1}, \cite{PPM2}, and \cite{HR}, \cite{BDM}. One also has the half partition algebra $P_{r+1/2}(\de)_B$, a sub-algebra of $P_{r+1}(\de)_B$.    Furthermore $P_r(n)$ and $P_{r+1/2}(n)$ act naturally on $V^{\otimes r}$ and these actions commute with the actions of $\Sym(n)$ and $\Sym(n-1)$. By  the result of Halverson-Ram, \cite[Theorem 3.6]{HR} the maps 

$$P_r(n)\to \End_{B\Sym(n)}(V^{\otimes r}) \eqno{(3)}$$
and

$$P_{r+1/2}(n)\to \End_{\Sym(n-1)}(V^{\otimes r})  \eqno{(4)}$$

given by the representations, are surjective.

\q Thus specialising to $m=n$ or $n-1$ in the above proposition we have the following, which is the main result of \cite{BDM}

\begin{corollary}   For any commutative ring $B$ and positive integers $n$ and  $r$,  the natural maps 
$$B\Sym(n)\to \End_{P_r(n)}(V^{\otimes r})$$
and 
$$B\Sym(n-1)\to \End_{P_{r+1/2}(n)}(V^{\otimes r})$$
are surjective.

\end{corollary}

\section{Further Examples}

\q We restrict to the classical case $q=1$. We fix a discrete valuation ring $R$ with field of fractions $K$ of characteristic $0$  and a  residue field $F$ of positive characteristic $p$.  Here \lq\lq base change" will mean base change from $R$ to $F$.  We shall say that base change holds for a finite $G$-set $X$, where $G$ is a finite group, if it holds on annihilators i.e. if the natural map $F\otimes_R \Ann_{RG}(RX) \to \Ann_{FG}(FX)$ is an isomorphism.

\q We give an example to show that base change and the double centraliser property may fail for a  Young permutation module  for a symmetric group.  Let $n$ be a positive integer.     For $\lambda\in \Par(n)$ we write simply $M(\lambda)$ for $M(\lambda)_F$ and  $\Sp(\lambda)$ for $\Sp(\lambda)_F$.

\q We use James's notation, \cite{James}, for the irreducible $F\Sym(n)$-modules.  A partition $\lambda=(\lambda_1,\lambda_2,\ldots)$ is called $p$-regular if  no non-zero part is repeated $p$-times. We write $\Par(n)_\reg$ for the set of $p$-regular partitions of $n$. For $\lambda\in \Par(n)_\reg$ the Specht module $\Sp(\lambda)$ has a unique irreducible quotient $D^\lambda$, and the modules $D^\lambda$, $\lambda\in \Par(n)_\reg$, form a complete set of pairwise non-isomorphic irreducible $F\Sym(n)$-modules.

\q For $\sigma\subseteq \Par(n)$ we write $C(\sigma)$ for the  cosaturation, i.e. for the set of $\mu\in \Par(n)$ such that $\mu\geq \lambda$ for some $\lambda\in\sigma$.

\q We fix a  Young $\Sym(n)$-set $X$, i.e. a finite $\Sym(n)$-set such that each point stabiliser is a Young subgroup. We write $\zeta(X)$ for the set of $\lambda\in \Par(n)$ such that $\Sym(n)$ is the stabiliser of some element of $X$. Now for $\lambda\in \Par(n)$,  the module $\Sp(\lambda)_K$ appears as a composition factor of $KX$ if and only if it appears as composition factor of $M(\lambda)_K$ for some $\lambda\in \zeta(X)$. However the composition factors of $M(\lambda)_K$ are precisely those $\Sp(\mu)_K$ with $\mu\geq \lambda$, see e.g. \cite[14.1]{James}. Hence the composition factors of $KX$ are precisely those $\Sp(\mu)_K$ with $\mu\in C(\zeta(X))$.  The  image of the  representation $\rho:KG\to \End_K(KX)$ has dimension 
$\sum_{\lambda\in C(\zeta(X)))} \dim(\lambda)^2$ and the annihilator of $KX$ has dimension 
$n! - \sum_{\lambda\in C(\zeta(X))} \dim(\lambda)^2$.      We define the $\Sym(n)$-set
$$C(X)=\coprod_{\lambda\in C(\zeta(X))} \Sym(n)/\Sym(\lambda).$$
 The above applies to $KC(X)$ so that the annihilator $KC(X)$ also has dimension $n! - \sum_{\lambda\in C(\zeta(X))} \dim(\lambda)^2$.  However, we know that the base change property on ideals holds for $C(X)$ so that $\Ann_{F\Sym(n)}(FC(X))$ also  has dimension  $n! - \sum_{\lambda\in C(\zeta(X))}\dim(\lambda)^2$. Now $\Ann_{F\Sym(n)}(FX)$ is the ideal of elements $a\in FG$ which act as zero on all $M(\lambda)$, $\lambda\in \zeta(X)$. Hence we have the following.

\begin{proposition}  Let $n$ be a non-negative integer and let $X$ be a Young $\Sym(n)$-set. Then  $\Ann_{F\Sym(n)}(FC(X))\leq \Ann_{F\Sym(n)}(FX)$ and $X$ has the base change property if and only if  $\Ann_{F\Sym(n)}(FC(X))= \Ann_{F\Sym(n)}(FX)$
\end{proposition}

\begin{remark} Let $\Lambda$ be a finite dimensional algebra over a  field $L$.  Let $M$ be a finite dimensional $\Lambda$-module. We will say that a $\Lambda$-module $N$ is a generalised section of $M$ if $N$ is a section of a finite direct sum of copies of $M$. Let $I$ be the annihilator of $M$. Then clearly $I$ annihilates any generalised section of $M$. Suppose  $N$ is a finite dimensional $\Lambda$-module annihilated by $I$. The representation $\rho:\Lambda\to \End_L(M)$ gives rise to an embedding of $\Lambda/I$ into a finite direct sum of copies of $M$.  Moreover,  there is an epimorphism from a finite direct sum of copies of $\Lambda/I$ to $N$ and so $N$ is a generalised section of $M$. Hence $N$ is annihilated by $I$ if and only if it is a generalised section of $M$.

\q Let $J$ be the Jacobson radical of $\Lambda$. If $J^kM=0$, for some positive integer $k$, then we have $J^kN=0$, for any generalised section of $N$ of $M$. Hence if $I$ annihilates $N$ then the Loewy length of $N$ is less than or equal to the Loewy length of $M$.
\end{remark}

\begin{proposition} Let $X$ be a Young permutation module for the symmetric group $\Sym(n)$. Then $X$ has the base change property on ideals  if and only if, for all $\mu\in C(\zeta(X))$,  the Young module $Y(\mu)$ is a generalised section of $FX$. 
\end{proposition}

\begin{proof} We form  $C(X) = \coprod_{\lambda\in C(\zeta(X))} \Sym(n)/\Sym(\lambda)$ as above.  Then, by Proposition 7.1,  $X$ has the base change property if and only if \\
$\Ann_{F\Sym(n)}(FX)=\Ann_{F\Sym(n)}(FC(X))$.  

\q Suppose $\mu\in C(\zeta(X))$. The Young module $Y(\mu)$ is a direct summand of $M(\mu)$ and hence annihilated by   
 $\Ann_{F\Sym(n)}(FC(X))$ and hence annihilated by 
 $\Ann_{F\Sym(n)}(FX)$. Hence $Y(\mu)$ is a generalised section of $FX$.

\q Now suppose that each $Y(\mu)$, with $\mu\in C(\zeta(X))$ is a generalised section of $FX$. Then $\Ann_{F\Sym(n)}(FX)$ annihilates every such $Y(\mu)$. However, $FC(X)$ is a direct sum of Young modules $Y(\mu)$ with $\mu\in C(\zeta(X))$.  Hence $\Ann_{F\Sym(n)}(FX)$ annihilates $FC(X)$, i.e. \\
$\Ann_{F\Sym(n)}(FX) \leq \Ann_{F\Sym(n)}(FC(X))$. The converse is certainly true so $\Ann_{F\Sym(n)}(FX)=\Ann_{F\Sym(n)}(FC(X))$.
Hence the dimension of \\
$\Ann_{F\Sym(n)}(FX)$ is the dimension of 
$\Ann_{F\Sym(n)}(FC(X))$ and, since $C(X)$ has the base change property, this is the $K$-dimension of \\
$\Ann_{K\Sym(n)}(KC(X))$, which is also the $K$-dimension of  
$\Ann_{F\Sym(n)}(KX)$ so we are done.
\end{proof}

\begin{corollary} Let $X$ be a Young $\Sym(n)$-set. 

(i) If $Y$ is also a Young $\Sym(n)$ set such that $\zeta(X)=\zeta(Y)$ then $X$ has base change if and only if $Y$ has base change. 

(ii) If base change fails for $X$ then there exists a minimal element $\lambda$ of $\zeta(X)$ such that base change fails for $\Sym(n)/\Sym(\lambda)$.

\end{corollary}

\q We now put these properties to  use in the example that we now give.  Since we have shown that base change and the double centraliser property holds for permutation module $M(\lambda)$ with $\lambda$ a hook the first case to look at here is $\lambda=(2,2)$. 

\q If the characteristic of $F$ is at least $5$ then the semisimplicity of of $M(\lambda)$ guarantees that the desired properties hold.

\q  Suppose now that the characteristic is $3$. Then $M(2,2)$ has image $M(4)$ so that $M(2,2)$ had the desired property if and only if $M(2,2) \oplus M(4)$ has. Moreover $M(2,2)$ has a filtration $0=V_0<V_1<V_2<V_3=M(2,2)$ with 
$$M(2,2)/V_1=\Sp(3,1)\oplus \Sp(4)=M(3,1)$$
(e.g. by block considerations)   so that $M(3,1)$ is an of $M(2,2)$  image and $M(2,2)$ has the desired properties since $M(2,2)\oplus M(3,1)\oplus M(4)$ has, by \cite[(1.4)]{LieModules}.

\q Now suppose  $F$ has characteristic $2$.  We shall use the Schur algebra  $S(4,4)$ and the Schur functor $f:\mod(S(4,4))\to \mod(F\Sym(4))$, as in \cite[Chapter 6]{EGS}.   Now the injective module  $I(2,2)$ has a filtration with sections $\nabla(2,2), \nabla(3,1),\nabla(4)$.  The module $I(2,2)$ is a direct summand  of the module $S^2E\otimes S^2E$ and this also has a filtration with sections $\nabla(2,2), \nabla(3,1),\nabla(4)$.  Hence we have $I(2,2)=S^2E\otimes S^2E$ and, applying the Schur functor, the module $M(2,2)=Y(2,2)$ has a filtration   $0=V_0<V_1<V_2<V_3=M(2,2)$ with $V_1=\Sp(2,2)$, $V_2/V_1=\Sp(3,1)$ and $V_3/V_2=\Sp(4)$.  Now $F\Sym(4)$ has simple modules $D^4=\Sp(4)$ and $D^{3,1}=\Sp(2,2)$.  Moreover the module $\Sp(3,1)$ has a submodule $U$ with $\Sp(3,1)/U=D^{3,1}$ and $U=D^4$.  The simple modules are self dual. The module $M(2,2)$ is self dual so that $D^{3,1}$ appears in the head as well as the socle. Moreover, since $M(2,2)$ is indecomposable the radical of $M(2,2)$ contains a copy of $D^{3,1}$. Similarly since $D^4$ appears in the head it also appears in the socle and indeed in the radical of $M(2,2)$. If follows that $M(2,2)$ has a submodule $Z$ with both $Z$ and $M(2,2)/Z$ isomorphic to  $D^{3,1}\oplus D^4$.  Hence $M(2,2)$ has Loewy length $2$. 

\q Now consider the module $M(3,1)$. The module $I(3,1)$ has sections $\nabla(3,1)$, $\nabla(4)$ as does the injective module $S^3E\otimes E$. Hence we have $I(3,1)=S^3E\otimes E$ and applying the Schur functor we have that $M(3,1)=Y(3,1)$ has a filtration with sections $\Sp(3,1)$, $\Sp(4)$. Since $M(3,1)$ is indecomposable it must be uniserial with sections $D^4,D^{3,1},D^4$. This module has Loewy length $3$ and so, by  Remark 7.2,  $M(3,1)$ can not be a generalised section of $M(2,2)$. Hence the annihilator of $M(2,2)$ does not annihilate $Y(3,1)$.  Hence, by Proposition 7.3, $\Sym(4)/\Sym(2,2)$ does not have the base change property.

\q In summary,  have shown that, taking $R$ a discrete valuation ring with residue field $F$ of characteristic $2$, then the  Young $\Sym(4)$-set \\
$\Sym(4)/\Sym(2,2)$ does not have the base change property, from $R$ to $F$,  on annihilators. By Lemma 3.5, the Young permutation module $M(2,2)$ does not have the double centraliser property.

%\newpage

\section{Annihilators and cell ideals}

\q We make some remarks on the annihilator ideals of Young permutation modules  in relation to  cell structure, in the spirit of \cite{Haerterich}, \cite{DotNym}, \cite{XX}.   

\q First consider a  semisimple algebra $B$,  finite dimensional over a field $F$. Let $S(\phi)$, $\phi\in \Phi$, be a complete set of pairwise non-isomorphic irreducible left $B$-modules.
For a subset $\tau$ of $\Phi$ we write $I(\tau)$ for the sum of all irreducible submodules,  of the left regular module $B$,  isomorphic  to $S(\phi)$ for some $\phi\in \Phi$. Then $I(\tau)$ is an ideal of $B$. We  have $B=I(\tau)\oplus I(\sigma)$, where $\sigma$ is the complement of $\tau$ in $\Phi$ and moreover: 

%\medskip

\sl

\phantom{hello} \hskip 30pt       $I(\tau)=\Ann_B(V)$ for any left  $B$-module $V$ of the form \\
\phantom{hello}   \hskip 7pt $\bigoplus_{\phi\in \Phi} S(\phi)^{(d_\phi)}$
  with $d_\phi\neq 0$ if and only if $\phi\in \sigma$. \rm   \hskip 80pt (1)
 
 \medskip

\q  We  recall the notion  of a cellular algebra due to Graham and Lehrer, \cite{GL}.

 \begin{definition}  Let $A$ be an algebra over a commutative ring $R$. A cell datum for   $(\Lambda,N,C,\,{}^*\,)$  for $A$ consists of the following.

\medskip

(C1) A partially ordered set $\Lambda$ and for each $\lambda\in \Lambda$ a finite set $N(\lambda)$ and an injective map $C: \coprod_{\lambda\in \Lambda} N(\lambda) \times N(\lambda) \to A$ with image an $R$-basis of $A$. 

(C2) For $\lambda\in \Lambda$ and $u,v\in N(\lambda)$ we write $C(u,v)=C^\lambda_{u,v}\in R$. Then   $^*$ is an $R$-linear involutory anti-automorphism  of $A$ such that $(C^\lambda_{u,v})^*= C^\lambda_{v,u}$.

(C3) If $\lambda\in \Lambda$ and $u,v\in N(\lambda)$ then for any element $a\in A$ we have 

$$aC^\lambda_{u,v}\equiv \sum_{u'\in N(\lambda)} r_a(u',u) C^\lambda_{u',v}   \hskip 20pt ({\rm mod } \  A(<\lambda))$$
where $r_a(u',u)\in R$ is independent of $v$ and where $A(<\lambda)$ is the $R$-submodule of $A$ generated by $\{ C^\mu_{u'',v''} \vert \mu\in  \Lambda, \mu<\lambda  \hbox{ and } u'',v''\in N(\mu)\}$.

\end{definition}

\q   We shall assume throughout that $\Lambda$ is finite. As in \cite{GL}, for $\lambda\in \Lambda$, we write  $W(\lambda)$ for the corresponding cell module. For $\tau \subseteq  \Lambda$ we write $A(\tau)$ for the $R$- span of all $C_{u,v}^\lambda$ with $\lambda\in \tau$ and $u,v\in N(\lambda)$.  Then $A/A(\tau)$ is a $R$-free on $\{C^\lambda_{u,v}+A(\tau) \vert \lambda\in \Lambda \backslash \tau,  u,v\in N(\lambda)\}$.   We shall write $W_R(\lambda)$ for $W(\lambda)$ and $A_R(\tau)$ for $A(\tau)$  if  we wish to emphasise the role of the base ring $R$.

\q Given a homomorphism of commutative rings $R\to \barR$ we obtain, by base change, a cell structure on $A_{\barR}=\barR\otimes_R A$, which  will also be denoted $(\Lambda,N,C,{}^*)$, see \cite[(1.8)]{GL}.  For $\tau\subseteq \Lambda$ and $\lambda\in \Lambda$, we identify $A(\tau)_\barR=\barR\otimes_R A(\tau)$ with $A_\barR(\tau)$ with  $W(\lambda)_\barR=\barR\otimes_R W(\lambda)$ with $W_\barR(\lambda)$ in the natural way.

\q Suppose for the moment that $R$ is a domain with field of fractions $F$.  For an $R$-submodule $X$ we write $X_F$ for $F\otimes_R X$ and if $X$ is a submodule of $A$ identify $X_F$ with a submodule of $A_F$.  Then we have
$$A(\tau)=A\bigcap A(\tau)_F \eqno{(2)}.$$

\q To  see this write  $X=A \bigcap A(\tau)_F$, the set of all $a\in A$ such that $ra\in A(\tau)$ for some $0\neq r\in R$. Thus $X/A(\tau)$ is $R$-torsion and since $A/A(\tau)$ is torsion free we have $X/A(\tau)=0$, i.e. $X=A(\tau)$.

\q We assume again  that $R$ is an arbitrary commutative ring. If   $\tau$ is a saturated subset of $\Lambda$ then $A(\tau)$ is an ideal of $A$,  called the cell ideal corresponding to $\tau$.

\q Suppose now that $R=F$ is a field and that $A$ is semisimple.  The cell modules $W(\lambda)$, $\lambda\in \Lambda$, form a complete set of pairwise non-isomorphic irreducible $A$-modules, see \cite[(3.8) Theorem]{GL}.  If $\tau$ is a saturated subset of $\Lambda$ with complement $\sigma$ then $A(\tau)$, as a left $A$-module has a filtration with quotients the cell modules  $W(\lambda)$, $\lambda\in \tau$,  and the quotient $A/A(\tau)$  has a filtration with sections the cell modules $A(\lambda)$, $\lambda\in \sigma$. Thus  from (1)  we have:

%\medskip
\sl

\phantom{hello} \hskip 20pt   $A(\tau)=I(\tau)$ and $I(\tau)=\Ann_A(V)$, where $V$ is any left module of the \\
\phantom{hello}   form    $\bigoplus_{\lambda\in \Lambda} W(\lambda)^{(d_\lambda)}$   with $d_\lambda\neq 0$ if and only if $\lambda\in \sigma$. \rm \hskip 50pt (3)
\medskip

\begin{lemma}   Let $R$ be a domain with field of fractions $F$  and $A$ an $R$-algebra with cell datum $(\Lambda,N,C,{}^*)$.  Suppose that $A_F$ is semisimple. Let $X$ be a left $A$-module, finitely generated and torsion free over $R$. Let $\sigma$ be a cosaturated subset of $\Lambda$.  If $X_F=\oplus_{\lambda\in \sigma} W_F(\lambda)^{(d_\lambda)}$, for positive integers $d_\lambda$,  then $\Ann_A(X)$ is the cell ideal $A(\tau)$, where $\tau$ is the complement of $\sigma$ in $\Lambda$.
\end{lemma}

\begin{proof}   For the ideal $I(\tau)$ of $A_F$, as defined above we have $\Ann_{A_F}(X_F)=I(\tau)$, by (1), and hence 
$$\Ann_A(X)=A\bigcap \Ann_{A_F}(X_F)=A\bigcap A_F(\tau)=A(\tau)$$
by (2).
\end{proof}

\q We are interested in taking for $A$ the Hecke algebra $\Hec(n)_{R,q}$, where $R$ is a commutative ring and $q$ a unit in $R$. Given a homomorphism of commutative rings $\theta:R\to {\bar R}$ we identify ${\bar R}\otimes_R \Hec(n)_{R,q}$ with $\Hec(n)_{{\bar R},\barq}$, where $\barq=\theta(q)$.  In particular we may take for the coefficient ring  $J=\zed[x,x^{-1}]$,  the ring of  Laurent polynomials with integer coefficients in the indeterminate $x$,  which has quotient field $K=\que(x)$, the field of rational functions.  Then $\Hec(n)_{K,x}$ is semisimple and  the Specht modules $\Sp(\lambda)_{K,x}$, $\lambda\in \Par(n)$, form a complete set of pairwise non-isomorphic irreducible modules. Thus we may take $\Lambda=\Par(n)$ in a cell structure.

\medskip

\q We shall say that a cell structure  $(\Lambda,N,C,{}^*)$  on $\Hec(n)_{J,x}$ is \emph{regular} if:

(i) $\Lambda=\Par(n)$, with partial order $\trianglelefteq$,  the dominance partial order; and

(ii) for $\lambda\in \Par(n)$ the corresponding cell module $W(\lambda)$ for the cell structure on $\Hec(n)_{J,x}$ is isomorphic to $\Sp(\lambda)_{J,x}$.

\q We shall say that a cell structure on $\Hec(n)_{R,q}$ (for a commutative ring $R$ and unit $q$ of $R$) is regular if it is obtained by base change from a regular cell structure on $\Hec(n)_{J,x}$.

\medskip

\q One has available two well known regular cell structures:  one coming from  the Kazhdan-Lusztig basis on  $\Hec(n)_{R,q}$, as in \cite{GL},   and the other from the  Murphy basis,  as in \cite{Murphy}.

\q We note that  some adjustments to the results that we will refer to from \cite{Murphy},  \cite{DotNym}, \cite{XX}  are  necessary because the statements  there are expressed in terms of right modules, and we are working primarily with left modules.  We leave it to the reader to make the appropriate modifications.  More importantly  we note that  in the definition of cell datum used in   \cite{DotNym}, \cite{XX} the ordering on $\Lambda$ given in condition \cite[(1.1) Definition, (C3)]{GL} has been reversed.     Thus we say that  a quadruple $(\Lambda,\tildeN,\tildeC,{}^*)$  is a datum of the second kind  for an algebra $A$ over a commutative ring $R$ if it satisfies conditions  $(C1),(C2)$  together with the following.

\medskip

$(\tildeC 3)$   If $\lambda\in \Lambda$ and $u,v\in  \tildeN(\lambda)$ then for any element $a\in A$ we have 

$$a \tildeC^\lambda_{u,v}\equiv \sum_{u'\in \tildeN(\lambda)} r_a(u',u) \tildeC^\lambda_{u',v}   \hskip 20pt ({\rm mod } \  A(>\lambda))$$
where $r_a(u',u)\in R$ is independent of $v$ and where $\tildeA(>\lambda)$ is the $R$-submodule of $A$ generated by $\{ \tildeC^\mu_{u'',v''} \vert \mu\in  \Lambda, \mu>\lambda  \hbox{ and } u'',v''\in \tildeN(\mu)\}$.

\medskip

\q Given a cell structure of the second kind on an algebra $A$ over a commutative ring $R$ and a subset $\sigma$ of $\Lambda$ we write $\tildeA(\sigma)$ for the $R$-span of all $\tildeC^\lambda_{u,v}$ with $\lambda\in \sigma$, $u,v\in \tildeN(\lambda)$.   If $\sigma$ is co-saturated then $\tildeA(\sigma)$ is an ideal, called a cell ideal. For $\lambda\in \Lambda$ and $v\in \tildeN(\lambda)$ the submodule of $\tildeA(\geq\lambda)/\tildeA(>\lambda)$ spanned by all $\tildeC^\lambda_{u,v}$, $u\in \tildeN(\lambda)$ has structure independent of $v$. (Here $\tildeA(\geq\lambda)=\tildeA(>\lambda)$, where $\rho=\{\mu\in \Par(n) \vert \mu\geq \lambda \}$.) Such a module is denoted  $\tildeW(\lambda)$ and called the cell module corresponding to $\lambda$.

\q We prefer slightly the outcome of the discussion below  when expressed in the    original formulation, \cite{GL}.  So we need to make some minor adjustments in passing between these kinds of cell structure. 

\q Let $R$ be a commutative ring and $q$ a unit. We shall consider cell structure on the Hecke algebra $\Hec(n)_{R,q}$.  We have three $R$-linear  involutions ${}^*$, ${}^\dagger$ and ${}^\sharp$ of the $R$-algebra $\Hec(n)_{R,q}$. The first two ${}^*$, ${}^\dagger$ are anti-automorphisms and the third ${}^\sharp$ is an automorphism.  They are given on generators by $T_i^*=T_i$ and $T_i^\dagger=T_i^\sharp= -T_i +(q-1)1$,  $1\leq i<n$.  These involutions commute and the composite  of any two is the third.

\q   Let $V$ be a  $\Hec(n)_{R,q}$-module,  finitely generated and free over $R$. We  denote by $V^*$ its dual, i.e.  the $R$-module $\Hom_R(V,R)$ 
regarded as a $\Hec(n)_{R,q}$-module via the action given by $a\alpha (v)=\alpha(a^*v)$, for $a\in \Hec(n)_{R,q}$, $\alpha\in V^*$, $v\in V$. We denote by  $V^\sharp$  the $R$-module $V$ regarded as a $\Hec(n)_{R,q}$-module via the action $a\cdot v=a^\sharp v$, for $a\in \Hec(n)_{R,q}$, $v\in V$.  For a subset $X$ of $\Hec(n)_{R,q}$ we set 
$X^\sharp=\{x^\sharp \vert x\in X\}$.  Note that $\Ann_{\Hec(n)_{R,q}}(V^\sharp)=\Ann_{\Hec(n)_{R,q}}(V)^\sharp$.

\q We briefly recall the cell structure  (of the second kind) for   $\Hec(n)_{R,q}$  coming from Murphy's paper, \cite{Murphy}.  For $\lambda\in \Par(n)$, we  write $\tildeN(\lambda)$ for the set of standard $\lambda$-tableaux.    For a composition $\alpha$ of $n$ we set  $x(\alpha)=\sum_{w\in \Sym(\alpha)} T_w$.  For $\lambda\in \Par(n)$ we  write $t^\lambda$ for the standard  $\lambda$ tableau with entries $1,2,\ldots,\lambda_1$ along the first row, entries $\lambda_1+1,\ldots,\lambda_1+\lambda_2$  along  the second row, etc.  For a row standard $\lambda$-tableau $s$ write $d(s)$ for the element $s\in \Sym(n)$ such that $s=d(s)\circ t^\lambda$.   For $s,t$ row standard $\lambda$ tableaux we have the element
$$x_{st}=T^*_{d(a)}x(\lambda) T_{d(b)}$$
of $\Hec(n)_{R,q}$. 
 
\q   Murphy shows that  $\bigcup_{\lambda\in \Par(n)} \{x_{uv} \vert u,v\in \tildeN(\lambda)\}$ is an $R$-basis  of $\Hec(n)_{R,q}$.      We have the map $\tildeC  :\coprod_{\lambda\in \Par(n)} \tildeN(\lambda) \times \tildeN(\lambda) \to \Hec(n)_{R,q}$, given by $\tildeC(s,t)=x_{st}$, for $\lambda\in \Par(n)$, $s,t\in \tildeN(\lambda)$.  Then $(\Par(n),\tildeN, \tildeC,{}^*)$ is a cell datum  (of the second kind) for  $\Hec(n)_{R,q}$ (where $\Par(n)$ is regarded as a poset via the dominance partial order).  The cell module $\tildeW(\lambda)=\tildeS^\lambda$ is isomorphic to the dual Specht module, i.e.  
 $\tildeW(\lambda)^*=(\tildeS^\lambda)^*$ is 
 the Specht module $\Sp(\lambda)_{R,q}$, by  \cite[Theorem 5.2 and Theorem 5.3]{Murphy}.   (In Murphy's paper the base ring is a domain, but the arguments are valid over any commutative ring  - alternatively one could appeal to Murphy's results directly for $R=J$, $q=x$ and apply base change.)

 \q We now consider the cell structure obtained by applying the automorphism ${}^\sharp$.     For $\lambda\in \Par(n)$ we define $N(\lambda)=\tildeN(\lambda')$.  For $\lambda\in \Par(n)$, $s,t\in N(\lambda)$ we define $C^\lambda_{st}=x_{st}^\sharp$.  Defining $C: \coprod_{\lambda\in \Par(n)} N(\lambda) \times N(\lambda) \to A=\Hec(n)_{R,q}$  by $C(s,t)=C^\lambda_{st}$, for $s,t\in N(\lambda)$, we have by \lq\lq transport of structure", that $(\Par(n),N,C,{}^*)$ is a cell structure on $A=\Hec(n)_{R,q}$.   Moreover, for $\lambda\in \Par(n)$, we have the cell module $W(\lambda)=\tildeW(\lambda')^\sharp=(\tildeS^{\lambda'})^\sharp $, which, by
  \cite[Theorem 5.2 and Theorem 5.3]{Murphy} is $\Sp(\lambda)_{R,q}$.  Clearly, the cell structure just given is obtained, by base change, from cell structure on $\Hec(n)_{J,x}$ constructed in the same way. Hence the structure given on $\Hec(n)_{R,q}$ is a regular cell structure. 
  
  \q From the definitions we have $\tildeA(\sigma)^\sharp=A(\sigma')$, for  a subset $\sigma$ of $\Par(n)$ (where $\sigma'=\{\lambda'\vert\lambda\in \sigma\}$).

\q We make  Lemma 8.2  more  explicit in the Hecke algebra case.

\begin{lemma}   Let $R$ be a domain, $q$ a unit and   $F$ the field of fractions of $R$.  We regard $\Hec(n)_{R,q}$ as a cellular algebra via a regular cell datum.   Suppose that 
$\Hec(n)_{F,q}$ is semisimple. Let $X$ be a left $\Hec(n)_{R,q}$-module, finitely generated and torsion free over $R$. Let $\sigma$ be a cosaturated subset of $\Par(n)$. If $X_F=\oplus_{\lambda\in \sigma} \Sp(\lambda)_{F,q}^{(d_\lambda)}$, for positive integers $d_\lambda$, then 
$$\Ann_{\Hec(n)_{R,q}}(X)=\Hec(n)_{R,q}(\tau)$$ 
where $\tau$ is the complement of $\sigma$ in $\Par(n)$.

\end{lemma}

\q We specialise further to the case in which the module $X$ as  above is a permutation module. For a subset $\sigma$ of $\Par(n)$ we write  $\sigma^{\trianglerighteq}$ for the set of $\mu\in \Par(n)$ such that $\mu\trianglerighteq\lambda$ for some $\lambda\in \sigma$.

\begin{lemma}   Let $R$ be a domain, $q$ a unit and   $F$ the field of fractions of $R$.   We regard $\Hec(n)_{R,q}$ as a cellular algebra via a regular cell datum.   Suppose that $\Hec(n)_{F,q}$ is semisimple. Let $X$ be a left $\Hec(n)_{R,q}$-module of the form \\
$\bigoplus_{\lambda\in \sigma} M(\lambda)_{R,q}^{(d_\lambda)}$ for a subset $\sigma$ of $\Par(n)$ with $d_\lambda\neq 0$ for all $\lambda\in\sigma$. Then 
$$\Ann_{\Hec(n)_{R,q}}(X)=\Hec(n)_{R,q}(\tau)$$ 
where $\tau$ is the complement of $\sigma^{\trianglerighteq}$. 
\end{lemma}

\begin{proof} By Young's rule,  \cite[14.1]{James},  the composition factors of $X_F$ are $\Sp(\lambda)_{F,q}$, $\lambda\in \sigma^{\trianglerighteq}$.    So the result follows from the previous lemma.
\end{proof}

\begin{remark} The main results of \cite{DotNym} are special cases with $\sigma=\{\lambda\}$ for some $\lambda\in \Par(n)$.   We get \cite[Theorem 5.2]{DotNym} by taking $R=F$ in the above lemma and we get \cite[Theorem 5.3]{DotNym} by taking $R=J,q=x$.  The result also gives,  in the classical case, the annihilator of a  Young permutation module over the integral symmetric group algebra $\zed \Sym(n)$ as a cell ideal  (taking $R=\zed$, $q=1$).
\end{remark}

\q We return to the the assumption that $R$ is an arbitrary commutative ring and $q$ a unit.   For a composition $\alpha$ we write $M(\alpha)_{R,q}$ for the left  $\Hec(n)_{R,q}$ induced from the trivial module (free over rank $1$ over $R$) for $\Hec(\alpha)_{R,q}$. Then $M(\alpha)_{R,q}$ may be realised as the left ideal $\Hec(n)_{R,q} x(\alpha)$ and is isomorphic to $M(\lambda)_{R,q}$, where $\lambda$ is the partition of $n$ obtained by arranging the parts of $\alpha$ in descending order. 

\begin{lemma} Let $\lambda,\mu\in \Par(n)$ and suppose that $\mu$ is a coarsening of $\lambda$. Then the annihilator of $M(\lambda)_{R,q}$ annihilates $M(\mu)_{R,q}$. 
\end{lemma}

\begin{proof}  Let $A=\Hec(n)_{R,q}$.  For some positive integer $m$ there  is a chain  $\lambda=\xi(0), \xi(1),\ldots, \xi(m)=\mu$ of elements of $\Par(n)$  such that  $\xi(i)$ is a simple coarsening of $\xi(i-1)$, $1\leq i\leq m$. If the annihilator of $M(\xi(i-1))_{R,q}$ is contained in the annihilator of $M(\xi(i))_{R,q}$, for $1\leq i\leq m$ then we have 
\begin{align*}\Ann_A(M(\lambda)_{R,q})&=\Ann_A(M(\xi(0))_{R,q})\subseteq \Ann_A(M(\xi(1)_{R,q})\cr
& \subseteq \cdots \subseteq \Ann_A(M(\xi(m)_{R,q})=\Ann_A(M(\mu)_{R,q}).
\end{align*}
Thus we may assume that $\mu$ is a simple coarsening of $\lambda$.  Writing \\ $\lambda=(\lambda_1,\lambda_2,\ldots)$, the partition  $\mu$ has the same parts as the composition \\
$\alpha=(\lambda_1,\ldots,\lambda_{i-1},\lambda_i+\lambda_{i+1},\lambda_{i+2},\ldots)$, for some $i$.  Now  \\
$\Hec(\lambda)_{R,q}\subseteq \Hec(\alpha)_{R,q}$ so we have a $\Hec(\lambda)_{R,q}$-module homomorphism $\phi:Rx(\lambda)\to M(\alpha)_{R,q}$ taking $x(\lambda)$ to $x(\alpha)$ and by Frobenius reciprocity this extends to an $A$-module homomorphism $\tilde \phi: M(\lambda)_{R,q}\to M(\alpha)_{R,q}$. Since $x(\alpha)$ is in the image the map $\tilde \phi$ is surjective and hence \\
$\Ann_A(M(\lambda)_{R,q})\subseteq \Ann_A(M(\alpha)_{R,q})$ and since $M(\alpha)_{R,q}$ is isomorphic to $M(\mu)_{R,q}$ we have  $\Ann_A(M(\lambda)_{R,q})\subseteq \Ann_A(M(\mu)_{R,q})$.  \end{proof}

\q Recall that for $\sigma\subseteq \Par(n)$ we are writing $\hatsigma$ for the coarsening of $\sigma$, i.e. the set of all partitions  $\mu$ such that $
\mu$ is a coarsening of some $\lambda\in \sigma$.

\begin{lemma}   Let $\sigma$ be a subset of $\Par(n)$ and let $X=\bigoplus_{\lambda\in \sigma} M(\lambda)_{R,q}^{(d_\lambda)}$, where $d_\lambda\neq 0$ for all $\lambda\in \sigma$. Then $\Ann_{\Hec(n)_{R,q}}(X)=\Ann_{\Hec(n)_{R,q}}({\hat X})$, where ${\hat X}=\bigoplus_{\lambda\in \hatsigma} M(\lambda)_{R,q}$. 
\end{lemma}

\begin{proof}   Let $A=\Hec(n)_{R,q}$. We must show that $\Ann_A(X)$ annihilates $M(\mu)_{R,q}$ for all $\mu\in \hatsigma$.  Now such $\mu$ is a coarsening of some $\lambda\in \sigma$ and hence $\Ann_A(X)\subseteq \Ann_A(M(\lambda)_{R,q})\subseteq \Ann_A(M(\mu)_{R,q})$, by Lemma 8.6. 
\end{proof}

\q Recall, from \cite{Murphy}, that we have on $\Hec(n)_{R,q}$ the $R$-bilinear symmetric bilinear form $\langle \ ,  \ \rangle: \Hec(n)_{R,q} \times \Hec(n)_{R,q} \to R$.  For $a,b\in \Hec(n)_{R,q}$, the value of   $\langle a , b \ \rangle$ is the coefficient of $1$ in $ab^*$,  expressed as an $R$-linear combination of the elements $T_w$, $w\in \Sym(n)$.   

\q If $s$ is a standard $\lambda$ tableau we write $[s]=\lambda$.   For $\lambda\in \Par(n)$, for $s$ a standard $\lambda$-tableau and $1\leq m\leq n$  we write $s\downarrow m$ for the the standard tableau whose Young diagram has boxes $(i,j)$ such that the $s(i,j)\in \{1,\ldots,m\}$,  and $(s\downarrow m)(i,j)=s(i,j)$.   As in \cite{Murphy},  we have the dominance partial order on $\lambda$ tableaux. Thus $s \trianglelefteq t$ if  $[s\downarrow m] \trianglelefteq [t\downarrow m]$ for all $1\leq m\leq n$.  For $\lambda$ tableaux $s,t,u,v$ we write $(s,t) \trianglelefteq (u,v)$ if $s \trianglelefteq u$ and $t\trianglelefteq v$.   Let $t$ be a $\lambda$-tableau. We write $t'$ for the transpose tableau, i.e. the $\lambda'$-tableau with $t'(i,j)=t(j,i)$, for $(i,j)$ in the Young diagram of $\lambda'$.   For standard $\lambda$-tableaux $s,t,u,v$ we have 
$$\langle x_{s't'}^\sharp, x_{uv}  \rangle = 0  \hbox{ unless }    (u,v) \trianglelefteq (s,t)\eqno{(4)}$$
and
$$\langle x_{s't'}^\sharp, x_{st}  \rangle = \pm q^N  \hbox{ for some positive integer   }     N  \eqno{(5)}$$
by \cite[Lemma 4.¤6]{Murphy}.

 \q Note that if $X=\bigoplus_{\lambda\in \Par(n)} M(\lambda)_{R,q}^{(d_\lambda)}$ then the $d_\lambda$ are uniquely determined by $X$.  (If $F$ is any residue field of $R$ and $\barq$ is the image of $q$ in $F$ then we then we have the  $\Hec(n)_{F,\barq}$-module   $X_F=\bigoplus_{\lambda\in \Par(n)} M(\lambda)_{F,\barq}^{(d_\lambda)}$  and so the $d_\lambda$ are uniquely determined by Lemma 5.6).   Thus we may extend the notation of Section 5 to Young permutation modules over $\Hec(n)_{R.q}$ by  putting $\zeta(X)=\{\lambda\in \Par(n) \vert d_\lambda\neq 0\}$.

\begin{proposition}    Let $X$ be a Young permutation module for $\Hec(n)_{R,q}$ and put $\sigma=\zeta(X)$.  If $\hatsigma$ is cosaturated then $\Ann_{\Hec(n)_{R,q}}(X)$ is the cell ideal $\Hec(n)_{R,q}(\tau)$, where $\tau$ is the complement of $\hatsigma$.
\end{proposition}

\begin{remark}  One gets the result for \lq\lq most " coefficient rings using the main results of this paper.  For example if $R=F[x,x^{-1}]$ is the ring of Laurent polynomials over a field $F$  then the result holds for $\Hec(n)_{R,x}(n)$ by Lemma 8.4, since $\Hec(n)_{F(x),x}$ is semisimple.   Moreover for each residue field $E$ of $F[x,x^{-1}]$ we have  that 
$E\otimes_R \Ann_{\Hec(n)_{R,x}}(X)\to \Ann_{\Hec(n)_{E,q}}(X_E)$  is an isomorphism, by Lemma 5.9.  Writing $X=\bigoplus_{\lambda\in \Lambda} M(\lambda)_{R,q}^{(d_\lambda)}$ and putting 
and ${\tilde X}=\bigoplus_{\lambda\in \sigma} M(\lambda)_{F[x,x^{-1}],x}$  we have, for any commutative $F$-algebra $R$ and unit $q\in R$, that  the natural map 
$R\otimes_{F[x,x^{-1}]} \Ann_{\Hec(n)_{F[x,x^{-1}],x}} ({\tilde X}) \to \Ann_{\Hec(n)_{R,q}}(X)$  is an isomorphism, by Corollary 2.13.  Hence

\begin{align*}\Ann_{\Hec(n)_{R,q}}(X)   & =R\otimes  _{F[x,x^{-1}]}    \Ann_{\Hec(n)_{F[x,x^{-1}],x}} ({\tilde X}) \cr
&=R\otimes_{F[x,x^{-1}]}    \Hec(n)_{F[x,x^{-1}],x}(\tau)\cr
&=\Hec(n)_{R,q}(\tau).
\end{align*}

\q Likewise if $q=1$ then $\Hec(n)_{R,q}=R\Sym(n)$ and the result holds for $\zed \Sym(n)$, since $\que \Sym(n)$ is semisimple.  But we have base change on annihilators by Proposition 5.9 and Lemma 3.5.  Thus the result holds for $R\Sym(n)$ for an arbitrary commutative ring $R$. 

\q However, we are unable to get the result for arbitrary $R,q$ by base change from the case $J=\zed[x,x^{-1}]$, $q=x$ using the results of Part I since we are assuming there that the base ring is a Dedekind domain (and $J$ is not.)

\q For    complete generality  we need to use  the argument of H\"arterich, \cite[Lemma 3]{Haerterich} (see also \cite[Proposition 7.2]{XX}).

\end{remark}

\begin{proof}  
We put $A=\Hec(n)_{R,q}$.   By base change and Lemma 8.4  we have $A(\tau)\subseteq \Ann_A(X)$.  

\q We must prove $\Ann_A(X)\subseteq A(\tau)$.  For a subset $\rho$ of $\Par(n)$ we put $X(\rho)=\bigoplus_{\lambda\in \rho} M(\lambda)_{R,q}$.  Clearly we may assume $X=X(\sigma)$.   Moreover by Lemma 8.7,  we may replace $\sigma$ by $\hatsigma$, i.e. we may  assume $\sigma$ is cosaturated.  We prove the result for cosaturated subsets $\sigma$ of $\Par(n)$ by induction on $|\sigma|$.  The result is clear if $\sigma$ is empty. Now suppose that $\sigma$ is not empty and the result holds for smaller cosaturated subsets.   

\q Let $\lambda$ be a minimal element of $\sigma$. By the inductive hypothesis the annihilator  of $X(\sigma\backslash\{\lambda\})$ is $A(\tau\bigcup\{\lambda\})=A(\tau)\oplus A(\lambda)$.    Thus we may write $h\in \Ann_A(X)$ in the
 form $h=h_1+h_2$ with 
 $h_1\in A(\tau)$, $h_2\in A(\lambda)$ and since $h_1\in \Ann_A(X)$ we have $h_2\in \Ann_A(X)$. Thus it suffices to prove that $h_2=0$, i.e.  that an element $h\in \Ann_A(X) \bigcap A(\lambda)$ is $0$.  We write $h=\sum_{s,t} r_{st} x_{s',t'}^\sharp$, with $s,t$ running over the standard  $\lambda$ tableaux and $r_{st}\in R$.   Now $h M(\lambda)_{R,q}=0$ i.e. 
   $hA x(\lambda)=0$ and so 
    $h  T_{w_1}^* x(\lambda) T_{w_2}=0$ for all $w_1,w_2\in \Sym(n)$.
     In particular  we have $h x_{uv}=0$
      for all standard $\lambda$ tableaux $u,v$ and hence 
 $h x_{uv}^*=h x_{vu}=0$ for all standard $\lambda$ tableaux $u,v$.  Thus  we have $\langle h, x_{uv} \rangle=0$ i.e.
 $$\sum_{s,t} r_{st} \langle x_{s't'}^\sharp, x_{uv} \rangle =0$$
 for all standard $\lambda$ tableaux $u,v$.
 
 \q Suppose, for a contradiction  that $h\neq 0$ and choose a pair of $\lambda$ tableaux  $(s_0,t_0)$ maximal (in the ordering $\trianglelefteq$) in the set 
 $\{  (s,t)\in \tildeN(\lambda) \vert r_{st}   \neq 0\}$.  Then  $\sum_{s,t} r_{st} \langle x_{s't'}^\sharp, x_{s_0t_0} \rangle =0$. But if a summand $r_{st} \langle x_{s't'}^\sharp, x_{s_0t_0} \rangle$ is non-zero then we have $r_{st}\neq 0$ and  $(s_0,t_0) \trianglelefteq (s,t)$, by (4),  and hence $(s,t)=(s_0,t_0)$. Hence we have 
 $$\sum_{s,t} r_{st} \langle x_{s't'}^\sharp, x_{s_0t_0} \rangle = r_{s_0t_0}  \langle x_{s_0't_0'}^\sharp, x_{s_0t_0} \rangle $$
 which, by (5), is  $\pm q^N r_{s_0t_0}$, for some positive integer $N$, and so $r_{s_0t_0}=0$, a contradiction.  Hence $h=0$ and $\Ann_A(X)=A(\tau)$. 
 \end{proof}

 \begin{example}   Let $H$ be a hook module for $\Hec(n)_{R,q}$, i.e. a finite direct sum $\bigoplus_\lambda M(\lambda)_{R,q}^{(d_\lambda)}$, with $\lambda$ running over hook partitions of $n$. Recall that the index $\idx(H)$ of $H$ is the smallest  positive integer $a$ such that $d_{(a,1^{n-a})}\neq 0$. 
Then putting   $\sigma=\{\lambda\in \Par(n) \vert d_\lambda\neq 0\}$ the set $\hatsigma$ consists of all 
  $\lambda=(\lambda_1,\lambda_2,\ldots)\in \Par(n)$ 
  such that   $\lambda_1\geq \idx(H)$, see Section 6.   Hence 
  $\Ann_{\Hec(n)_{R,q}}(H)$ is the cell ideal $\Hec(n)_{R,q}(\tau)$, where $\tau$ is the set of partitions $\lambda=(\lambda_1,\lambda_2,\ldots)$ such that $\lambda_1<\idx(H)$.

\q In particular   we can take  $q=1$. Let $V$ be a free module of rank $n$ over an arbitrary commutative ring $R$, with $R$-basis $v_1,\ldots,v_n$.  We regard $V$,      and hence  the $r$-fold tensor product
$V^{\otimes r}=V\otimes_R \otimes \cdots \otimes_R  V$  as a permutation module for 
$A=R\Sym(n)$.  Then $V^{\otimes r}$ is the permutation module on the $\Sym(n)$-set $I(n,r)$.  We regard $V^{\otimes r}$ as an $R\Sym(m)$-module, for $m\leq n$, by restriction.  Then from  the above   (or Remark 8.9)   $\Ann_{R\Sym(m)}(V^{\otimes r})$ is the cell ideal $A(\tau)$, where $\tau$ is the set of all $\lambda=(\lambda_1,\lambda_2,\ldots)\in \Par(m)$ such that  $\lambda_1< \idx(I(n,r)|_{\Sym(m)})$, i.e.  (by Section 6, (2))  $\lambda_1< m-\min\{m,r\}$.    That this holds for  $m=n,n-1$ and 
with $R$ a commutative domain is the main result Theorem 7.5 of    \cite{XX}.  

\end{example}

%\newpage

%\input biblio.tex
%\end{document}

%biblio.tex

%biblio.tex

\end{document}